\documentclass[a4paper]{amsart}
\usepackage{url}
\usepackage[all]{xy}
\usepackage[mathscr]{eucal}
\usepackage{amsmath,amssymb,amsfonts}
\usepackage{mathrsfs,latexsym,amsthm,enumerate}
\newtheorem{theorem}{Theorem}[section]
\newtheorem{lemma}[theorem]{Lemma}

\newtheorem{example}[theorem]{Example}

\newtheorem{proposition}[theorem]{Proposition}
\newtheorem{remark}[theorem]{Remark}

\title{The Booleanization of an inverse semigroup}

\author{Mark V. Lawson}
\address{Mark V.Lawson, Department of Mathematics
and the
Maxwell Institute for Mathematical Sciences, 
Heriot-Watt University,
Riccarton,
Edinburgh EH14 4AS, 
UNITED KINGDOM}
\email{m.v.lawson@hw.ac.uk}

\begin{document} 

%%%%%%%%%%%%%%%%%%%%%%%%%%%%%%%%%%%%%%%%%%%%%%%%%%%%%%%%%%%%%%%%%%%%%%%%%%%%%%%%%%%%%%%%%%%%%%
\begin{abstract} 
We prove that the forgetful functor from the category of Boolean inverse semigroups to the category of inverse semigroups with zero has a left adjoint.
This left adjoint is what we term the `Booleanization'.
We establish the exact theoretical connection between the Booleanization of an inverse semigroup and Paterson's universal groupoid of the inverse semigroup
and we explicitly compute the concrete Booleanization of the polycyclic inverse monoid $P_{n}$ and demonstrate its affiliation with the Cuntz-Toeplitz algebra.
\end{abstract}
\maketitle

%%%%%%%%%%%%%%%%%%%%%%%%%%%%%%%%%%%%%%%%%%%%%%%%%%%%%%%%%%%%%%%%%%%%%%%%%%%%%%%%%%%%%%%%%%%%%%%%%
\section{Introduction}

The goal of this paper is to prove the following theorem and to study some of its consequences.

\begin{theorem}[Booleanization]\label{them:bc}
Let $S$ be an inverse semigroup with zero.
Then there is a Boolean inverse semigroup $\mathsf{B}(S)$ together with an embedding $\beta \colon S \rightarrow \mathsf{B}(S)$
such that if $\theta \colon S \rightarrow T$ is any homomorphism to a Boolean inverse semigroup $T$ then there is a unique morphism $\gamma \colon \mathsf{B}(S) \rightarrow T$
such that $\theta = \beta \gamma$.
\end{theorem}

Although we originally constructed $\mathsf{B}(S)$ in \cite{LL}, we restricted the class of homomorphisms considered
and so did not prove universality in full generality.
This lacuna will be filled here.

We refer the reader to \cite{Law1} for background on inverse semigroups,  to \cite{Resende} for background on \'etale groupoids
and to  \cite{Wehrung2} for the theory of Boolean inverse semigroups.
All distributive lattices will be assumed to have a bottom but not necessarily a top;
if they have a top we say they are {\em unital}.
We use the term {\em Boolean algebra} to mean what is often termed a `generalized Boolean algebra'.
Thus a Boolean algebra is a distributive lattice in which each principal order ideal is a unital Boolean algebra.
The inverse semigroups in this paper will usually have a zero and for those that do homomorphisms between them will be required to preserve zero.
The order on inverse semigroups will be the {\em natural partial order}.
The semilattice of idempotents of an inverse semigroup $S$ is denoted by $\mathsf{E}(S)$.
More generally, if $X$ is any subset of $S$ then $\mathsf{E}(X) = \mathsf{E}(S) \cap X$.
In addition, define
$$X^{\uparrow} = \{s \in S \colon \exists x \in X, x \leq s \}
\text{ and }
X^{\downarrow} = \{s \in S \colon \exists x \in X, s \leq x \}.$$
If $X = \{x\}$ then we write simply $x^{\uparrow}$ and $x^{\downarrow}$, respectively.
If $s$ is an element of an inverse semigroup define $\mathbf{d}(s) = s^{-1}s$ and $\mathbf{r}(s) = ss^{-1}$.
The {\em compatibility relation} $\sim$ in an inverse semigroup  
is defined by $s \sim t$ if and only if $s^{-1}t$ and $st^{-1}$ are both idempotents.
A set that consists of elements which are pairwise compatible is said to be {\em compatible}.
The {\em orthogonality relation} $\perp$ in an inverse semigroup with zero 
is defined by $s \perp t$ if and only if $s^{-1}t = 0 = st^{-1}$.
Observe that  $s \perp t$ 
if and only if 
$\mathbf{d}(a) \perp \mathbf{d}(b)$ and $\mathbf{r}(a) \perp \mathbf{r}(b)$.
A set that consists of elements which are pairwise orthogonal is said to be {\em orthogonal}.
The proof of the following is straightforward.

\begin{lemma}\label{lem:jersey} Let $S$ be an inverse semigroup with zero and let $s,t,a \in S$.
If $s \perp t$ then both $sa \perp ta$ and $as \perp at$.
\end{lemma} 

\begin{lemma}\label{lem:buffs} In an arbitrary inverse semigroup, let $a \sim b$.
Then the following are equivalent:
\begin{enumerate}
\item $a \perp b$.
\item $\mathbf{d}(a) \perp \mathbf{d}(b)$.
\item $\mathbf{r}(a) \perp \mathbf{r}(b)$.
\end{enumerate}
\end{lemma}
\begin{proof} By symmetry, it is enough to prove the equivalence of (1) and (2).
Clearly, (1) implies (2).
We prove that (2) implies (1).
To do this, we need to prove that $\mathbf{r}(a) \perp \mathbf{r}(b)$.
We are given that $\mathbf{d}(a) \perp \mathbf{d}(b)$, which is equivalent to $ab^{-1} = 0$,
and and also given that $a^{-1}b$ is an idempotent.
Therefore $aa^{-1}bb^{-1} = a(a^{-1}b)b^{-1} \leq ab^{-1} = 0$.
We have proved that $\mathbf{r}(a) \perp \mathbf{r}(b)$ and so $a \perp b$.
\end{proof}

An inverse semigroup is said to be a {\em $\wedge$-semigroup} if and only if it has all binary meets.
An inverse semigroup is said to be {\em distributive} if it has binary joins of compatible elements and multiplication distributes over such joins.
A {\em pseudogroup} is an inverse semigroup with zero in which every compatible subset has a join
and where multiplication distributes over all such joins.
Observe that the semilattice of idempotents of a pseudogroup is a {\em frame} in the sense of \cite{J}.
A distributive inverse semigroup is {\em Boolean} if its semilattice of idempotents is a Boolean algebra.
A {\em morphism} between distributive inverse semigroups is a homomorphism of inverse semigroups with zero
that maps binary compatible joins to binary compatible joins.
A {\em $\wedge$-morphism} between distributive $\wedge$-semigroups preserves binary meets.

Proofs of the following can be found in \cite[Lemma~1.4.11, Lemma~1.4.12, Lemma~1.4.14]{Law1}.

\begin{lemma}\label{lem:macron} Let $S$ be an inverse semigroup.
\begin{enumerate}
\item $s \sim t$ if and only if $s \wedge t$ exists and 
$\mathbf{d}(s \wedge t) = \mathbf{d}(s)\mathbf{d}(t)$ 
and
$\mathbf{r}(s \wedge t) = \mathbf{r}(s)\mathbf{r}(t)$.
\item If $s \sim t$ then $s \wedge t = st^{-1}t = ss^{-1}t$.
\item $s^{\downarrow}$ is a compatible set.
\end{enumerate}
\end{lemma}

\noindent
{\bf Definition.} If $s \sim t$ then the meet $s \wedge t$ guaranteed by part (1) of Lemma~\ref{lem:macron} is called a
{\em compatible meet}.  Observe  by part (2) of  Lemma~\ref{lem:macron} that such meets can be constructed purely algebraically. \\

\begin{remark}\label{rem:lepen}{\em Morphisms between distributive inverse semigroups are required to preserve binary compatible joins
but are not required to preserve any binary meets that might exist.
However, compatible meets are preserved by any homomorphism.
This simple observation will prove useful in establishing the universality of our construction. }
\end{remark}

%%%%%%%%%%%%%%%%%%%%%%%%%%%%%%%%%%%%%%%%%%%%%%%%%%%%%%%%%%%%%%%%%%%%%%%%%%%%%%%%%%%%%%%%%%%%%%%%%%%%%%%%%%%%%%%%%%%%%%%%%%%%%%%%%%%%%%
We shall need a little more on the algebraic properties of Boolean inverse semigroups.
Let $B$ be a Boolean algebra.
If $e,f \in B$ are such that $f \leq e$ then there is a unique element $e \setminus f$ such that
$e = (e \setminus f) \vee f$ and $(e \setminus f) \perp f$. 
In fact, $e \setminus f$ is the largest element of $B$ that is less than $e$ and orthogonal to $f$.
Now let $S$ be an inverse semigroup whose semilattice of idempotents forms a Boolean algebra.
If $a,b \in S$ with $b \leq a$, define $a \setminus b = a(\mathbf{d}(a) \setminus \mathbf{d}(b))$.
Observe that $b \perp (a \setminus b)$ by Lemma~\ref{lem:buffs}.
If $S$ is, in fact, a Boolean inverse semigroup then $a = b \vee (a \setminus b)$;
in this case, $a \setminus b$ is the largest element less than or equal to $a$ which is orthogonal to $b$.
The proof of the following is straightforward.

\begin{lemma}\label{lem:trump} Let $S$ be a Boolean inverse monoid.
Let $b \leq a$.
Suppose that $x \leq a$, $x \perp b$ and $a = b \vee x$.
Then $x = a \setminus b$.
\end{lemma}

The following shows how to write a compatible join of two elements as a binary orthogonal join of two elements.

\begin{lemma}\label{lem:gove} Let $a$ and $b$ be compatible elements in an inverse semigroup whose semilattice of idempotents forms a Boolean algebra under the natural partial order.
\begin{enumerate}
\item $a \setminus (a \wedge b)$ and $b$ are orthogonal.
\item $a \vee b$ exists if and only if  $(a \setminus (a \wedge b)) \vee b$ exists in which case they are equal.
\end{enumerate}
\end{lemma}
\begin{proof} (1) The element $a \wedge b$ exists by Lemma~\ref{lem:macron}.
Thus $a \setminus (a \wedge b)$ exists by Lemma~\ref{lem:trump}.
Clearly, $\mathbf{d}(a \setminus (a \wedge b)) \mathbf{d}(b) = 0$.
By Lemma~\ref{lem:buffs}, it follows that $(a \setminus (a \wedge b)) \perp b$.

(2) It is enough to prove that $a,b \leq c$ if and only if  $(a \setminus (a \wedge b)), b \leq c$.
Only one direction needs an actual proof.
Suppose that $(a \setminus (a \wedge b)), b \leq c$.
We need to prove that $a \leq c$.
But this follows from the fact that $a = (a \setminus (a \wedge b)) \vee (a \wedge b)$
by Lemma~\ref{lem:trump}.\end{proof}

The following result is frequently invoked in proofs and follows easily from Lemma~\ref{lem:gove}

\begin{proposition}\label{prop:definition} The following are equivalent.
\begin{enumerate}
\item $S$ is a Boolean inverse semigroup.
\item $S$ has all binary orthogonal joins, multiplication distributes over such joins, 
and its semilattice of idempotents forms a Boolean algebra with respect to the natural partial order.
\end{enumerate}
\end{proposition}

%%%%%%%%%%%%%%%%%%%%%%%%%%%%%%%%%%%%%%%%%%%%%%%%%%%%%%%%%%%%%%%%%%%%%%%%%%%%%%%%%%%%%%%%%%%%%%%%%%%%%%%%%%%%%%%%%%%%%%
A useful first step in the proof of Theorem~\ref{them:bc} is the observation that we may restrict to the case where the inverse semigroup $S$ is actually a distributive inverse semigroup.
This follows from the fact that the forgetful functor from the category of distributive inverse semigroups to
the category of inverse semigroups has a left adjoint which is easy to construct.
We recall this construction here.
Let $S$ be an inverse semigroup.
The {\em Schein completion} of $S$  \cite[Theorem 1.4.23, Theorem 1.4.24]{Law1} is the {\em pseudogroup} $\mathsf{C}(S)$ whose elements are the compatible order ideals of $S$.
Multiplication is subset multiplication; the order is subset inclusion; the idempotents are order ideals of $\mathsf{E}(S)$.
In fact, $\mathsf{C}$ is left adjoint to the forgetful functor from the category of pseudogroups to the category of inverse semigroups.
We now describe the `finite' elements of $\mathsf{C}(S)$.
Let $T$ be a pseudogroup.
An element $a \in T$ is said to be {\em finite} if $a \leq \bigvee_{i \in I} b_{i}$ implies that there is a finite subset $\{1, \ldots, n\} \subseteq I$
such that  $a \leq \bigvee_{i =1}^{n} b_{i}$.
The set of finite elements of $T$ is denoted by $\mathsf{K}(T)$.
The following is a slightly sharper statement of \cite[Lemma~3.3]{LL}.

\begin{lemma} Let $S$ be a pseudogroup.
\begin{enumerate}
\item $a$ is finite if and only if $a^{-1}$ is finite.
\item If $a$ is any element and $e$ is a finite idempotent $e \leq a^{-1}a$ (respectively, $e \leq aa^{-1}$) then $ae$ (respectively, $ea$) is finite.
\item $a$ is finite if and only if $a^{-1}a$ is finite (respectively, $aa^{-1}$ is finite).
\item If $a$ and $b$ are finite and $\mathbf{d}(a) = \mathbf{r}(b)$ then $ab$ is finite.
\end{enumerate}
\end{lemma}

The following is \cite[Lemma~3.4]{LL}.

\begin{lemma}\label{lem:cathy} Let $S$ be a pseudogroup.
\begin{enumerate}
\item The finite elements of $S$ form an inverse subsemigroup if and only if the finite idempotents form a
subsemilattice.
\item If the finite elements form an inverse subsemigroup they form a distributive inverse subsemigroup.
\end{enumerate}
\end{lemma}

Define $\mathsf{D}(S) = \mathsf{K}(\mathsf{C}(S))$.
An element $A$ of $\mathsf{C}(S)$ is said to be {\em finitely generated} if $A = \{a_{1}, \ldots, a_{m}\}^{\downarrow}$ where
$\{a_{1}, \ldots, a_{m}\}$ is a finite compatible subset of $S$.

\begin{lemma}\label{lem:election} Let $S$ be an inverse semigroup.
\begin{enumerate}
\item The finite elements in $\mathsf{C}(S)$ are precisely the finitely generated ones.
\item  $\mathsf{D}(S)$ is a distributive inverse semigroup and there is an embedding $\delta \colon S \rightarrow \mathsf{D}(S)$ given by $s \mapsto s^{\downarrow}$.
\end{enumerate}
\end{lemma}
\begin{proof} (1) It is immediate from the definition that finitely generated elements of $\mathsf{C}(S)$ are finite.
We prove the converse.
Let $A$ be a finite element of $\mathsf{C}(S)$. 
Then $A \leq \bigvee_{s \in A} s^{\downarrow}$.
Thus $A \leq \bigvee_{i=1}^{n} s_{i}^{\downarrow}$ for some finite set of elements $\{s_{i} \colon 1 \leq i \leq n \} \subseteq A$.
Clearly, $A = \bigcup_{i=1}^{n} s_{i}^{\downarrow}$ and so $A$ is finitely generated.

(2) By Lemma~\ref{lem:cathy}, it is enough to prove that the product of two finite idempotents is finite.
This is straightforward.
\end{proof}

\begin{theorem}[Distributive completion]\label{theorem:dc} Let $S$ be an inverse semigroup.
There is a distributive inverse semigroup $\mathsf{D}(S)$ together with an embedding $\delta \colon S \rightarrow \mathsf{D}(S)$
such that if $\theta \colon S \rightarrow T$ is any homomorphism to a distributive inverse semigroup $T$ there is a unique morphism $\gamma \colon \mathsf{D}(S) \rightarrow T$
such that $\theta = \delta \gamma$.
\end{theorem}
\begin{proof} Let $\theta \colon S \rightarrow T$ be a homomorphism to a distributive inverse semigroup $T$.
We prove that there is a unique morphism $\gamma \colon \mathsf{D}(S) \rightarrow T$ such that $\theta = \phi \gamma$.
Define 
$$\gamma ( \{a_{1}, \ldots, a_{n}\}^{\downarrow} ) = \bigvee_{i=1}^{n} \theta (a_{i}).$$ 
Observe that the right hand side above is defined because $\{a_{1}, \ldots, a_{n} \}$ is a compatible subset of $S$ and homomorphisms preserve compatibility.
We prove that $\gamma$ is well-defined.
Suppose that  $\{a_{1}, \ldots, a_{n}\}^{\downarrow} =  \{b_{1}, \ldots, b_{m}\}^{\downarrow}$.
Then for each $1 \leq i \leq n$, we can write $a_{i} = b_{i}'$ where $b_{i}' \leq b_{j}$ for some $j$.
It follows that $\bigvee_{i=1}^{n} \theta (a_{i}) \leq \bigvee_{j=1}^{m} \theta (b_{j})$. 
By symmetry, we get equality and so $\gamma$ is well-defined.
The proofs of the remaining parts of the theorem are now routine.
\end{proof}

\begin{remark}\label{rem:berk} {\em Let $S$ be an inverse semigroup and let $\theta \colon S \rightarrow T$ 
be a homomorphism to a Boolean inverse semigroup $T$.
Then by Theorem~\ref{theorem:dc}, there is a unique morphism $\theta' \colon \mathsf{D}(S) \rightarrow T$
such that $\theta' \delta = \theta$.
Thus we can construct our Booleanization first for distributive inverse semigroups.}
\end{remark}

\begin{remark}\label{rem:monoid}{\em Suppose that $S$ is an inverse monoid.
Then $1^{\downarrow}$ is a finitely generated compatible order ideal.
It follows that $\mathsf{D}(S)$ is a monoid and that $\delta$ is also a monoid homomorphism.}
\end{remark}

\begin{example}\label{ex:guppy}
{\em It is tempting to believe that if $S$ were already a distributive inverse semigroup then it should be isomorphic to $\mathsf{D}(S)$.
We construct a simple counterexample to show that this is not true.
Let $S = I_{2}$, the symmetric inverse monoid on a set with 2 elements.
This is a (finite) distributive inverse monoid.
Let $A$ be the compatible order ideal of all elements below the transposition $1 \leftrightarrow 2$.
This contains four elements.
Let $B$ be the compatible order ideal generated by the partial bijections $1 \mapsto 2$ and $2 \mapsto 1$.
This contains 3 elements.
Evidently, $A,B \in \mathsf{D}(I_{2})$ and $\bigvee A = \bigvee B$ but $A \neq B$.
It follows that the map $\mathsf{D}(I_{2}) \rightarrow I_{2}$ given by $X \mapsto \bigvee X$ is surjective but not injective.
Thus, in particular, $I_{2}$ is not isomorphic to $\mathsf{D}(I_{2})$.}
\end{example}

An inverse semigroup is said to be a {\em weak semilattice} \cite{Steinberg} if the intersection of any two principal order ideals is finitely generated.

\begin{lemma}\label{lem:lenin} Let $S$ be an inverse semigroup.
Then $S$ is a weak semilattice if and only if $\mathsf{D}(S)$ is a $\wedge$-semigroup.
\end{lemma} 
\begin{proof} The natural partial order in $\mathsf{D}(S)$ is subset inclusion.
Suppose first that $S$ is a weak semilattice.
Let $A,B \in \mathsf{D}(S)$ where $A = \{a_{1}, \ldots, a_{m} \}^{\downarrow}$ and $B = \{b_{1}, \ldots, b_{n} \}^{\downarrow}$.
For each pair $(i,j)$, where $1 \leq i \leq m$ and $1 \leq j \leq n$,
choose a finite set of generators $C_{i,j}$ of the order ideal $a_{i}^{\downarrow} \cap b_{j}^{\downarrow}$.
Then it is easy to see that $A \cap B = (\bigcup_{(i,j)} C_{i,j})^{\downarrow}$.
It follows that $\mathsf{D}(S)$ is a $\wedge$-semigroup.
Conversely, suppose that $\mathsf{D}(S)$ is a $\wedge$-semigroup.
Let $a,b \in S$.
Then $a^{\downarrow}, b^{\downarrow} \in \mathsf{D}(S)$.
By assumption, $a^{\downarrow} \wedge b^{\downarrow}$ exists.
Thus we can write $a^{\downarrow} \wedge b^{\downarrow} = \{c_{1}, \ldots, c_{m} \}^{\downarrow} \in \mathsf{D}(S)$.
It follows that $\{c_{1}, \ldots, c_{m} \}^{\downarrow} \subseteq a^{\downarrow} \cap b^{\downarrow}$.
Let $c \leq a,b$.
Then $c^{\downarrow} \subseteq a^{\downarrow},b^{\downarrow}$.
Thus $c^{\downarrow} \subseteq \{c_{1}, \ldots, c_{m} \}^{\downarrow}$ and so $c \in \{c_{1}, \ldots, c_{m} \}^{\downarrow}$.  
It follows that $a^{\downarrow} \cap b^{\downarrow} = \{c_{1}, \ldots, c_{m} \}^{\downarrow}$,
and therefore $S$ is a weak semilattice. 
\end{proof}

%%%%%%%%%%%%%%%%%%%%%%%%%%%%%%%%%%%%%%%%%%%%%%%%%%%%%%%%%%%%%%%%%%%%%%%%%%%%%%%%%%%%%%%%%%%%%%%%%%%
Let $S$ be an inverse semigroup.
A {\em filter} in $S$ is a subset $A$ such that $A = A^{\uparrow}$ and whenever $a,b \in A$ there exists $c \in A$ such that $c \leq a,b$.
A filter is {\em proper} if it does not contain zero.
In what follows, any results stated about filters are proved in \cite{Law3, Law5, LL}.
Observe that $A$ is a filter if and only if $A^{-1}$ is a filter.
If $A$ and $B$ are filters then $(AB)^{\uparrow}$ is a filter.
Define $\mathbf{d}(A) = (A^{-1}A)^{\uparrow}$ and $\mathbf{r}(A) = (AA^{-1})^{\uparrow}$.
Then both $\mathbf{d}(A)$ and $\mathbf{r}(A)$ are filters.
It is easy to check that $A$ is proper if and only if $\mathbf{d}(A)$ is proper (respectively, $\mathbf{r}(A)$ is proper).
Observe that for each $a \in A$ we have that $A = (a \mathbf{d}(A))^{\uparrow} = (\mathbf{r}(A)a)^{\uparrow}$.
We denote the set of proper filters on $S$ by $\mathcal{L}(S)$.\footnote{Observe that this notation refers only to proper filters.}
If $A,B \in \mathcal{L}(S)$, define the partial operation $A \cdot B$ if and only if $\mathbf{d}(A) = \mathbf{r}(B)$
in which case $A \cdot B = (AB)^{\uparrow}$.
In this way, $\mathcal{L}(S)$ becomes a groupoid in which
the identities are the filters that contain idempotents; indeed, these are precisely the filters that are also inverse subsemigroups. 

Let $S$ be a distributive inverse semigroup.
A {\em prime filter} in $S$ is a proper filter $A \subseteq S$ such that if $a \vee b \in A$
then $a \in A$ or $b \in A$.
Denote the set of all prime filters of $S$ by $\mathsf{G}(S)$.
It can be checked that $A$ is a prime filter if and only if  $\mathbf{d}(A)$ (respectively, $\mathbf{r}(A)$) is a prime filter.
Define a partial multiplication $\cdot$ on $\mathsf{G}(S)$ by $A \cdot B$ exists if and only if $\mathbf{d}(A) = \mathbf{r}(B)$,
in which case $A \cdot B = (AB)^{\uparrow}$.
With respect to this partial multiplication,  $\mathsf{G}(S)$ is a groupoid where
the identities are the prime filters that contain idempotents.
For this reason, it is convenient to define a prime filter to be an {\em identity} if it contains an idempotent.
Proofs of all of the above claims can be found in \cite{LL}.

Let $G$ be any discrete groupoid with set of identities $G_{o}$.
A subset $X \subseteq G$ is said to be a {\em partial bisection} if $x,y \in X$ and $\mathbf{d}(x) = \mathbf{d}(y)$ then $x = y$,
and if $x,y \in X$ and $\mathbf{r}(x) = \mathbf{r}(y)$ then $x = y$.
It is easy to check that a subset $X \subseteq G$ is a partial bisection precisely when $X^{-1}X,XX^{-1} \subseteq G_{o}$.
A partial bisection is called simply a {\em bisection} if in fact 
$X^{-1}X = G_{o} = XX^{-1}$.\footnote{Partial bisections appear under a wondrous multitude of aliases.
I have settled on this terminology because then partial bisections correspond to partial bijections and bisections correspond to bijections.
The word `local' is a loaded one in topology and has the wrong connotations.}
Denote the set of all partial bisections of $G$ by $\mathsf{L}(G)$.
Endow it with the binary operation of subset multiplication.
The following is well-known \cite[page 12]{Resende}.

\begin{proposition}\label{prop:lb} Let $G$ be a discrete groupoid.
Then $\mathsf{L}(G)$ is a pseudogroup in which the natural partial order is subset inclusion.
\end{proposition}

Let $S$ be a distributive inverse semigroup.
For each $a \in S$ define $V_{a}$ to be the set of all prime filters that contains $a$.
The following is proved in \cite{LL}.

\begin{proposition}\label{prop:key} Let $S$ be a distributive inverse semigroup.
\begin{enumerate}
\item If $a \nleq b$ then there is a prime filter that contains $a$ and omits $b$.
\item $V_{a} \subseteq V_{b}$ if and only if $a \leq b$.
\item $V_{a}$ consists entirely of identities if and only if $a$ is an idempotent.
\item The sets $V_{a}$ form a basis for a topology with respect to which they are compact-open partial bisections.
In this way,  $\mathsf{G}(S)$ becomes an \'etale groupoid who space of identities is a {\em locally compact} spectral space
---
this means that it is sober and has a basis of compact-open sets closed under binary intersections.
\item If $a \sim b$ then $V_{a} \cup V_{b} = V_{a \vee b}$.
\end{enumerate}
\end{proposition}

The philosophy that underlies this paper can be usefully summarized by the following table:

\begin{center}
\begin{tabular}{|c||c|}\hline
{\bf Commutative} & {\bf Non-commutative} \\ \hline \hline
{\small Meet semilattice} & {\small Inverse semigroup} \\ \hline
{\small Frame} & {\small Pseudogroup} \\ \hline
{\small Distributive lattice} & {\small Distributive inverse semigroup} \\ \hline
{\small Boolean algebra} & {\small Boolean inverse semigroup} \\ \hline
\end{tabular}
\end{center}

%%%%%%%%%%%%%%%%%%%%%%%%%%%%%%%%%%%%%%%%%%%%%%%%%%%%%%%%%%%%%%%%%%%%%%%%%%%%%%%%%%%%%%%%%%%%%
\section{Proof of Theorem~\ref{them:bc}}

This splits into two parts.
First,  we construct the Boolean inverse semigroup $\mathsf{B}(S)$ and the map $\beta \colon S \rightarrow \mathsf{B}(S)$;
this reiterates what we showed in \cite{LL}.
Second, which is the new part,  we prove that $\beta$ has the requisite universal properties.

If $S$ is an inverse semigroup then we can replace it by $\mathsf{D}(S)$ when calculating its Booleanization by Remark~\ref{rem:berk}. 
Thus in what follows we shall assume that $S$ is distributive.
We present a direct construction of the Booleanization of $S$ in the case where $S$ is simply an inverse semigroup with zero in Section~2.4.

We introduce some important notation.
Let $S$ be a distributive inverse semigroup.
If $b \leq a$ where $a,b \in S$  define
$V_{a;b} = V_{a} \setminus V_{b}$.
Since this is a subset of a partial bisection it is itself a partial bisection.
Observe that $V_{a;b} \neq \varnothing$ if and only if $b < a$ by part (1) of Proposition~\ref{prop:key}.
If $e \in \mathsf{E}(S)$ then we denote by $\mathcal{V}_{e}$ the set of prime filters in $\mathsf{E}(S)$ containing $e$.
If $f \leq e$  define $\mathcal{V}_{e;f} = \mathcal{V}_{e} \setminus \mathcal{V}_{f}$.\\

\noindent
{\bf Notation.} Keep clear the typographical distinction between
$V_{a}$, which is a set of prime filters {\em in} $S$, and $\mathcal{V}_{e}$, which is the set of prime filters {\em in} $\mathsf{E}(S)$.\\ 

The properties of the operation $(a,b) \mapsto a \setminus b$ in a Boolean inverse semigroup motivate this paper.
They are summarized below.
Observe that whenever we write $s \setminus t$ we assume that $t \leq s$.

\begin{lemma}\label{lem:properties} Let $S$ be a Boolean inverse semigroup.
\begin{enumerate}
\item $(s \setminus t)^{-1} = s^{-1} \setminus t^{-1}$.
\item If $a$ is any element then $(s \setminus t)a = sa \setminus ta$.
\item $s \setminus (u \vee v) = (s \setminus u)s^{-1}(s \setminus v)$.
\item $(s \setminus t)(u \setminus v) = su \setminus (sv \vee tu)$.
\item $V_{a;b} = V_{c;d}$ if and only if $a \setminus b = c \setminus d$.
\item If $a \leq b \leq c$ then $(c \setminus b) \leq (c \setminus a)$.
\end{enumerate}
\end{lemma} 
\begin{proof}
(1) Straightforward.

(2) Observe that from $t \leq s$ we obtain $ta \leq sa$.
From $s = t \vee (s \setminus t)$ we get that $sa = ta \vee (s \setminus t)a$
and from $t \perp (s \setminus t)$ we get that $ta \perp (s \setminus t)a$.
It follows by Lemma~\ref{lem:trump} that $(s \setminus t)a = sa \setminus ta$.

(3) This follows from the fact that in a Boolean algebra $e \setminus (i \vee j) = (e \setminus i)(e \setminus j)$.

(4) Observe that $sv,tu \leq su$ and so $sv \sim tu$ meaning that the join $sv \vee tu$ exists in $S$.
We have the following argument, where we use parts (3) and (2).
\begin{eqnarray*}
su \setminus (sv \vee tu) &=& (su \setminus tu) u^{-1}s^{-1} (su \setminus sv)\\
                                         &=& (s \setminus t) uu^{-1} s^{-1}s (u \setminus v)\\
                                         &=& (s \setminus t) s^{-1}s uu^{-1} (u \setminus v)\\
                                         &=& (s \setminus t)(u \setminus v).  
\end{eqnarray*}

(5) We prove that in a Boolean inverse semigroup $V_{a;b} = V_{a \setminus b}$.
The result then follows by part (2) of Proposition~\ref{prop:key}.
Let $A \in V_{a;b}$.
Then $a \in A$ and $b \notin A$.
But $a = b \vee (a \setminus b)$ and $A$ is a prime filter.
Thus $a \setminus b \in A$.
In the other direction, let $A \in V_{a \setminus b}$.
Then $a \in A$ since $a \setminus b \leq a$.
But $A$ is a prime filter and $a = b \vee a \setminus b$
so that $b \in A$ or $a \setminus b \in A$.
However $b \in A$ and $a \setminus b \in A$ implies that $0 = b \wedge (a \setminus b) \in A$, which is impossible.
Thus $a \setminus b \in A$, as required.

(6) This follows from the fact that $\mathbf{d}(a) \leq \mathbf{d}(b) \leq \mathbf{d}(c)$ and that
$(\mathbf{d}(c) \setminus \mathbf{d}(b)) \leq (\mathbf{d}(c) \setminus \mathbf{d}(a))$.
\end{proof}

%%%%%%%%%%%%%%%%%%%%%%%%%%%%%%%%%%%%%%%%%%%%%%%%%%%%%%%%%%%%%%%%%%%%%%%%%%%%%%%%%%%%%%%%%%%%%%
\subsection{Construction}

Let $S$ be a distributive inverse semigroup.
We begin by constructing the discrete groupoid $\mathsf{G}(S)$ of prime filters on $S$ and then the pseudogroup $\mathsf{L}(\mathsf{G}(S))$ by Proposition~\ref{prop:lb}.
By Proposition~\ref{prop:key} there is an embedding $\iota \colon S \rightarrow \mathsf{L}(\mathsf{G}(S))$ given by $\iota (a) = V_{a}$.
We shall construct $\mathsf{B}(S)$ as an inverse subsemigroup of  $\mathsf{L}(\mathsf{G}(S))$ that contains the image of $\iota$.

\begin{lemma}\label{lem:benn} Let $S$ be a distributive inverse semigroup.
\begin{enumerate}
\item $V_{s;t}^{-1} = V_{s^{-1};t^{-1}}$.
\item $V_{s;t}V_{u;v} = V_{su;sv \vee tu}$.
\item Let $a \sim b$, $c \sim d$ and $c \vee d \leq a \vee b$.  Then $V_{(a \vee b);(c \vee d)} = V_{a;(c \vee d)a^{-1}a} \cup V_{b; (c \vee d)b^{-1}b}$.
\end{enumerate}
\end{lemma}
\begin{proof} (1) The proof follows by two simple observations.
First, $A$ is a prime filter if and only if $A^{-1}$ is a prime filter.
Second, $A$ is a prime filter that contains $s$ and omits $t$ if and only if $A^{-1}$ is a prime filter that contains $s^{-1}$ and omits $t^{-1}$.

(2) This is proved as part (1) of \cite[Proposition 5.5]{LL} though in this paper we have taken the opportunity to make the obvious simplification
in the statement of the result.

(3) Let $P \in V_{(a \vee b);(c \vee d)}$.
Then $a \vee b \in P$ and $c \vee d \notin P$.
Suppose that $a \in P$.
If $(c \vee d)a^{-1}a \in P$ then $c \vee d \in P$ which is a contradiction.
By a similar argument we deduce that 
$V_{(a \vee b);(c \vee d)} \subseteq V_{a;(c \vee d)a^{-1}a} \cup V_{b; (c \vee d)b^{-1}b}$.
Now let $P \in V_{a;(c \vee d)a^{-1}a}$.
Then $a \in P$ and $(c \vee d)a^{-1}a \notin P$.
It follows that $a \vee b \in P$.
Suppose that $c \vee d \in P$.
Then $(c \vee d)a^{-1}a \in P$ which is a contradiction.
It follows that $P \in V_{(a \vee b);(c \vee d)}$.
By a similar argument we have proved the claim.
\end{proof}

By Lemma~\ref{lem:benn}, the set $\mathsf{V}(S) = \{V_{s;t} \colon t < s\}$ forms an inverse subsemigroup of $\mathsf{L}(\mathsf{G}(S))$.
More generally, let $V$ be an inverse subsemigroup of a distributive inverse semigroup $L$.
Define $V^{\vee}$ to be the set of all elements of $L$ which are compatible joins of finite compatible subsets of $V$.
The proof of the following is straightforward.

\begin{lemma}\label{lem:fillon} Let $V$ be an inverse subsemigroup of a distributive inverse semigroup $L$.
Then  $V^{\vee}$ is a distributive inverse semigroup.
\end{lemma}

In the light of Lemma~\ref{lem:fillon}, define $\mathsf{B}(S)$ to be the set of all unions within $\mathsf{L}(\mathsf{G}(S))$ of finite compatible subsets of $\mathsf{V}(S)$ 
and define $\beta \colon S \rightarrow \mathsf{B}(S)$ by $\beta (s) = V_{s}$.
Then $\mathsf{B}(S)$ is a distributive inverse semigroup and $\beta$ is a morphism of distributive inverse semigroups.

\begin{lemma}\label{lem:dypso} Let $S$ be a distributive inverse semigroup.
\begin{enumerate}

\item Let $b < a$ and $d < c$.
If $V_{a;b} \subseteq V_{c;d}$ then $V_{\mathbf{d}(a);\mathbf{d}(b)} \subseteq V_{\mathbf{d}(c);\mathbf{d}(d)}$. 

\item  Let $f < e$ and $j < i$ be idempotents.
Then $V_{e;f} \subseteq V_{i;j}$ if and only if $\mathcal{V}_{e;f} \subseteq \mathcal{V}_{i;j}$. 

\end{enumerate}
\end{lemma}
\begin{proof} (1)  Let $F \in  V_{\mathbf{d}(a);\mathbf{d}(b)}$.
Then $A = (aF)^{\uparrow} \in V_{a;b}$.
Thus $A \in V_{c;d}$.
But then $F = \mathbf{d}(A) \in  V_{\mathbf{d}(c);\mathbf{d}(d)}$. 

(2) The crux of the proof is that if $A$ is a prime filter that contains an idempotent then $A = \mathsf{E}(A)^{\uparrow}$ where $\mathsf{E}(A)$ is a prime
filter in the distributive lattice of idempotents.
Suppose that $V_{e;f} \subseteq V_{i;j}$.
Let $F \in \mathcal{V}_{e;f}$.
Then $F^{\uparrow} \in V_{e,f}$.
Thus $F^{\uparrow} \in V_{i,j}$.
It follows that $F \in \mathcal{V}_{i;j}$.
Conversely, suppose that $\mathcal{V}_{e;f} \subseteq \mathcal{V}_{i;j}$. 
Let $A \in V_{e;f}$.
Then $\mathsf{E}(A) \in \mathcal{V}_{e;f}$ and so  $\mathsf{E}(A) \in \mathcal{V}_{i;j}$.
Thus $A \in V_{i;j}$.
\end{proof}

\begin{lemma}\label{lem:cake} Let $S$ be a distributive inverse semigroup.
If $V_{a;b}$ is such that every element is an identity prime filter then 
$V_{a;b} = V_{\mathbf{d}(a); \mathbf{d}(b)}$.
\end{lemma}
\begin{proof} If $a = b$ the result is trivial.
Let $A \in V_{a,b}$.
Then since $A$ is an identity prime filter, $\mathbf{d}(A) = A$.
It follows that $\mathbf{d}(a) \in A$.
Suppose that $\mathbf{d}(b) \in A$.
Then $b = a\mathbf{d}(b) \in A$, since $A$ is an identity.
This is a contradiction.
Thus $\mathbf{d}(b) \notin A$.
Hence $A \in V_{\mathbf{d}(a);\mathbf{d}(b)}$.
It follows that $V_{a;b} \subseteq V_{\mathbf{d}(a);\mathbf{d}(b)}$.
Let $F \in V_{\mathbf{d}(a);\mathbf{d}(b)}$.
Then $A = (aF)^{\uparrow}$ is a prime filter.
Clearly, $A \in V_{a}$.
If $b \in A$ then $\mathbf{d}(b) \in F$, which is a contradiction.
Thus $A \in V_{a;b}$.
But by what we proved above $\mathbf{d}(a) \in A$.
Thus there is an idempotent $e \in A$ such that $e \leq \mathbf{d}(a),a$.
It follows that $e \in F$ and so $a \in F$.
We have proved that $A = F$ and so $F \in V_{a;b}$.
\end{proof}

\begin{proposition}\label{prop:marx} Let $S$ be a distributive inverse semigroup.
Then $\mathsf{B}(S)$ is a Boolean inverse semigroup and $\beta$ is a morphism.
\end{proposition}
\begin{proof} It only remains to prove that the idempotents of  $\mathsf{B}(S)$ form a Boolean algebra.
The idempotents in $\mathsf{B}(S)$ are those partial bisections which are subsets of the identity space of the groupoid $\mathsf{G}(S)$.
They are therefore finite joins of elements of the form $V_{a;b}$ where the only elements of $V_{a;b}$ are identity prime filters.
By Lemma~\ref{lem:cake}, we have that $V_{a;b} = V_{\mathbf{d}(a); \mathbf{d}(b)}$.
It now follows by Lemma~\ref{lem:dypso} that there is an order isomorphism between $\mathsf{E}(\mathsf{B}(S))$ and $\mathsf{B}(\mathsf{E}(S))$
induced by the map  $V_{e;f} \mapsto \mathcal{V}_{e; f}$.
That $\mathsf{B}(\mathsf{E}(S))$ is a Boolean algebra can be easily verified, indeed,
it is the Booleanization of the distributive lattice $\mathsf{E}(S)$ \cite[Proposition~II.4.5]{J}.
This gives us the result.
\end{proof}

\begin{remark}\label{rem:more-monoid}
{\em Let $S$ be a distributive inverse monoid.
Then $V_{1}$ is the identity of $\mathsf{B}(S)$ and the map $\beta$ is a monoid homomorphism.}
\end{remark}

%%%%%%%%%%%%%%%%%%%%%%%%%%%%%%%%%%%%%%%%%%%%%%%%%%%%%%%%%%%%%%%%%%%%%%%%%%%%%%%%%%%%%%%%%%%%%%%%
\subsection{Universality}

We have described the semigroup $\mathsf{B}(S)$ in terms of prime filters.
To prove the universality of this construction, we need to convert certain results about prime filters into purely algebraic and order-theoretic results.
The motivation for doing this came from \cite[Section~2.3]{Wehrung1}.
Our first lemma exemplifies our approach.

\begin{lemma}\label{lem:portmanteau} Let $D$ be a distributive lattice and let $e,f,i,j \in D$ such that $f < e$ and $j < i$.
Then $\mathcal{V}_{e ;f} \subseteq \mathcal{V}_{i;j}$ 
if and only if 
$e = f \vee (e \wedge i)$ and $f \wedge j = e \wedge j$.
\end{lemma}
\begin{proof} Suppose first that  $e = f \vee (e \wedge i)$ and $f \wedge j = e \wedge j$.
Let $F \in  \mathcal{V}_{e ;f}$
Then $e \in F$ and $f \notin F$.
It follows that $e \wedge i \in F$ and so $i \in F$.
Suppose that $j \in F$.
Then $e \wedge j \in F$ and so $f \wedge j \in F$ meaning that $f \in F$, which is a contradiction.
It follows that $j \notin F$.

To prove the converse, let $\mathcal{V}_{e ;f} \subseteq \mathcal{V}_{i;j}$.
Clearly, $f \vee (e \wedge i) \leq e$.
Suppose that $e \nleq f \vee (e \wedge i)$.
Then there is a prime filter $F$ such that $e \in F$ and $f \vee (e \wedge i) \notin F$.
Clearly, $f \notin F$.
Thus, by assumption, $i \in F$ and $j \notin F$.
But then $e,i \in F$ and so $e \wedge i \in F$ implying that  $f \vee (e \wedge i) \in F$, which is a contradiction.
It follows that $e = f \vee (e \wedge i)$.
Clearly, $f \wedge j \leq e \wedge j$.
Suppose that $e \wedge j \nleq f \wedge j$.
Then there is a prime filter $F$ such that $e \wedge j \in F$ and $f \wedge j \notin F$.
Clearly, $f \notin F$ and so $F \in \mathcal{V}_{e;f}$.
Thus $i \in F$ and $j \notin F$ but this contradicts the fact that $j \in F$.
It follows that  $e \wedge j = f \wedge j$.
\end{proof}

%%%%%%%%%%%%%%%%%%%%%%%%%%%%%%%%%%%%%%%%%%%%%%%%%%%%%%%%%%%%%%%%%%%%%%%%%%%%%%%%%%%%%%%%%%%%%%%%%%%%%%%%%%%%%
\begin{lemma}\label{lem:valery} Let $S$ be a distributive inverse semigroup
and let $a,b,c,d \in S$ be such that $b < a$ and $d < c$.
\begin{enumerate}

\item $V_{a;b} \subseteq V_{c;d}$ if and only if 
$\mathcal{V}_{\mathbf{d}(a); \mathbf{d}(b)} \subseteq  \mathcal{V}_{\mathbf{d}(c); \mathbf{d}(d)}$ 
and there exists $x \leq c$ such that $a = b \vee x$.  

\item $V_{a;b} \subseteq V_{c;d}$ if and only if there exist $d \leq b' < a' \leq c$ such that $V_{a;b} =  V_{a';b'}$.

\end{enumerate}
\end{lemma}
\begin{proof} (1) Suppose that $V_{a;b} \subseteq V_{c;d}$. 
Then $V_{a} \subseteq V_{c} \cup V_{b}$.
Thus $V_{a} = (V_{a} \cap V_{c}) \cup V_{b}$.
Since $V_{a} \cap V_{c}$ is open (and non-empty) we may cover it by means of sets of the form $V_{x_{i}}$ where $i \in I$.
If we adjoin $V_{b}$ we thereby have a cover of $V_{a}$.
Since $V_{a}$ is compact we may write $(V_{a} \cap V_{c}) \cup V_{b} = \left( \bigcup_{i =1}^{m} V_{x_{i}} \right) \cup V_{b}$
for some finite subset $\{1, \ldots, m\} \subseteq I$.
We have that $\bigcup_{i =1}^{m} V_{x_{i}} \subseteq V_{a} \cap V_{c}$.
Thus the elements $x_{i}$ are pairwise compatible and so have a join $x$, say, where $x \leq a,c$.
It follows that $V_{a} = V_{x} \cup V_{b}$.
Since $x,b \leq a$ the elements $x$ and $b$ are compatible and so $a = x \vee b$.
The fact that $\mathcal{V}_{\mathbf{d}(a); \mathbf{d}(b)} \subseteq  \mathcal{V}_{\mathbf{d}(c); \mathbf{d}(d)}$ 
follows by Lemma~\ref{lem:dypso}.

To prove the converse,  
suppose that  $V_{\mathbf{d}(a); \mathbf{d}(b)} \subseteq  V_{\mathbf{d}(c); \mathbf{d}(d)}$ 
and there exists a non-zero $x \leq c$ such that $a = b \vee x$.  
Let $A \in V_{a;b}$.
Then by Lemma~\ref{lem:dypso}, we have that $\mathbf{d}(A) \in V_{\mathbf{d}(a); \mathbf{d}(b)}$ and so by
assumption $\mathbf{d}(A) \in V_{\mathbf{d}(c); \mathbf{d}(d)}$.
By assumption, $a \in A$ and so since $b \notin A$ we have that $x \in A$ and so $c \in A$.
Suppose that $d \in A$.
Then $\mathbf{d}(d) \in \mathbf{d}(A)$ which is a conradiction.
It follows that $A \in V_{c;d}$. 

(2) By part (1) there exists $x \leq c$ such that $a = b \vee x$.
The meet $b \wedge x$ exists since $b,x \leq a$ by Lemma~\ref{lem:macron}.
Likewise, the meet $d \wedge x$ exists since $d,x \leq c$.
But $b \wedge x, d \wedge x \leq x$ and so $(b \wedge x) \vee (d \wedge x)$ exists.
It is now routine to check that
$$V_{a;b} = V_{x;(x \wedge b) \vee (x \wedge d)}.$$
Observe that $x \leq c$.
In what follows, we may therefore assume, without loss of generality, that in fact $a \leq c$.
It is now immediate that $b \vee d \leq a \vee d$ with both joins existing.
It is routine to check that
$$V_{a;b} = V_{(a \vee d);(b \vee d)}.$$
But
$$d \leq b' = b \vee d < a' = a \vee d \leq c.$$
\end{proof}

The following result is the first step in proving universality.

\begin{proposition}\label{prop:gaulle} Let $S$ be a distributive inverse semigroup
and let $\alpha \colon S \rightarrow T$ be a morphism to a distributive inverse semigroup $T$.
If $V_{a;b} \subseteq V_{c;d}$ then  $V_{\alpha(a);\alpha(b)} \subseteq V_{\alpha(c);\alpha(d)}$.
\end{proposition}
\begin{proof} We deal first with the special case where $V_{a;b} = V_{c;d}$, we have that
$V_{\mathbf{d}(a); \mathbf{d}(b)} =  V_{\mathbf{d}(c); \mathbf{d}(d)}$ 
and there exist $x \leq c$ such that $a = b \vee x$ 
and $y \leq a$ such that $c = d \vee y$ 
by Lemma~\ref{lem:valery}.
By  Lemma~\ref{lem:dypso} we have that $\mathcal{V}_{\mathbf{d}(a); \mathbf{d}(b)} =  \mathcal{V}_{\mathbf{d}(c); \mathbf{d}(d)}$. 
Lemma~\ref{lem:portmanteau} tells us that $\mathcal{V}_{\mathbf{d}(a); \mathbf{d}(b)} =  \mathcal{V}_{\mathbf{d}(c); \mathbf{d}(d)}$
is equivalent to purely lattice-theoretic conditions and so 
$\mathcal{V}_{\alpha (\mathbf{d}(a)); \alpha(\mathbf{d}(b))} =  \mathcal{V}_{\alpha(\mathbf{d}(c)); \alpha(\mathbf{d}(d))}$.
By Lemma~\ref{lem:dypso}, we therefore have that 
$V_{\alpha (\mathbf{d}(a)); \alpha(\mathbf{d}(b))} =  V_{\alpha(\mathbf{d}(c)); \alpha(\mathbf{d}(d))}$.
From $a = b \vee x$ where $x \leq a,c$
and $c = d \vee y$ where $y \leq a,c$
we get that
$\alpha (a) = \alpha (b) \vee \alpha (x)$ where $\alpha (x) \leq \alpha (a),\alpha (c)$
and $\alpha (c) = \alpha (d) \vee \alpha (y)$ where $\alpha (y) \leq \alpha (a), \alpha (c)$.
By Lemma~\ref{lem:valery} again,
it follows that  $V_{\alpha(a);\alpha(b)} = V_{\alpha(c);\alpha(d)}$.

We can now deal with the case where $V_{a;b} \subseteq V_{c;d}$.
By Lemma~\ref{lem:valery}, there exist $d \leq b' < a' \leq c$ such that $V_{a;b} =  V_{a';b'}$.
But by the special case proved above,
$V_{\alpha (a);\alpha (b)} =  V_{\alpha (a');\alpha (b')}$.
Now $\alpha (d) \leq \alpha (b') < \alpha (a') \leq \alpha (c)$.
It is routine to check that
$V_{\alpha (a');\alpha (b')} \subseteq V_{\alpha(c);\alpha(d)}$.
Thus 
 $V_{\alpha(a);\alpha(b)} \subseteq V_{\alpha(c);\alpha(d)}$,
as required.
\end{proof}

\begin{lemma}\label{lem:garl} Let $S$ be a distributive inverse semigroup.
Let $a,b,a_{i},b_{i} \in S$, where $1 \leq i \leq m$, be such that $b \leq b_{i} \leq a_{i} \leq a$,
and where $\{V_{a_{i};b_{i}} \colon 1 \leq i \leq m\}$ is a compatible subset of $\mathsf{B}(S)$.
Then
$$V_{a;b} = \bigcup_{i=1}^{m} V_{a_{i};b_{i}}$$
if and only if the following three conditions hold:
\begin{enumerate}
\item $a = \bigvee_{i=1}^{m} a_{i}$.
\item $b = \bigwedge_{i=1}^{m} b_{i}$, a compatible meet.
\item For all $X \subseteq \{1, \ldots, m\} = Y$, where $X \neq \emptyset$ and $X \neq Y$, we have that 
$$\bigwedge_{i \in X} b_{i} \leq \bigvee_{j \in Y \setminus X} a_{j},$$
where $\bigwedge_{i \in X} b_{i}$ is a compatible meet. 
\end{enumerate}
\end{lemma}
\begin{proof} Observe that since $a_{i},b_{i} \leq a$ for all $1 \leq i \leq m$ the set $\{a_{i}, b_{i} \colon 1 \leq i \leq m\}$ is compatible.
Thus all joins and meets which are needed actually exist. 

Suppose first that  $V_{a;b} = \bigcup_{i=1}^{m} V_{a_{i};b_{i}}$.
We prove that (1), (2) and (3) all hold.
(1) holds. Clearly $a \geq \bigvee_{i=1}^{m} a_{i}$.
To prove the reverse inequality,
let $P$ be a prime filter that contains $a$ but omits $\bigvee_{i=1}^{m} a_{i}$.
Then $P$ must omit $b$.
It follows that  $a_{i} \in P$ for some $i$ but this is a contradiction.
The proof that (2) holds is similar.
(3) holds.
Suppose that $X \subseteq Y$ where $X \neq \emptyset$ and $X \neq Y$.
Let $P$ be a prime filter that contains $\bigwedge_{i \in X}b_{i}$ but omits $\bigvee_{j \in Y \setminus X} a_{j}$.
Let $i \in X$.
Then $b_{i} \in P$ and so $a \in P$.
But if $b \in P$ then all $a_{i}$ would be $P$.
Thus $P$ contains $a$ and omits $b$.
By assumption, $P$ contains $a_{k}$ but omits $b_{k}$, for some $k$.
Now $k \notin Y \setminus X$ and so $k \in X$.
But then by assumption $b_{k} \in P$ which is a contradiction.

We now assume that (1), (2) and (3) all hold.
We prove that $V_{a;b} = \bigcup_{i=1}^{m} V_{a_{i};b_{i}}$.
Let $P \in V_{a;b}$.
Then $P$ contains $a$ and omits $b$.
By (1), $P$ must contain at least one $a_{i}$ and by (2), $P$ must omit at least one $b_{j}$.
Let $X \subseteq \{1, \ldots, m \} = Y$ be the proper, non-empty subset of {\em all} $i \in Y$ such that $b_{i} \in P$.
Then $\bigwedge_{i \in X} b_{i} \in P$ and so by (3), we must have that $\bigvee_{j \in Y \setminus X} a_{j} \in P$.
Thus for some $j$ we have that $a_{j} \in P$.
But, by assumption, $b_{j} \notin P$.
It follows that $P \in V_{a_{j} ; b_{j} }$.
That the reverse inclusion holds is immediate.
\end{proof}

\begin{proposition}\label{prop:britannia} 
Let $\alpha \colon S \rightarrow T$ be a morphism between distributive inverse semigroups.
Define $\gamma \colon \mathsf{B}(S) \rightarrow \mathsf{B}(T)$ by
$$\beta \left(  \bigcup_{i=1}^{m} V_{a_{i} ; b_{i}} \right) 
= 
 \bigcup_{i=1}^{m} V_{\alpha (a_{i}) ; \alpha (b_{i})}$$
where $\{V_{a_{i};b_{i}} \colon 1 \leq i \leq m\}$ is a compatible subset of $\mathsf{B}(S)$.
 Then $\beta$ is well-defined.
\end{proposition}
\begin{proof} The proof in the case $i = 1$ follows by Proposition~\ref{prop:gaulle}.

Next consider the case where
$$V_{a ; b}   
=
\bigcup_{i=1}^{m} V_{a_{i} ; b_{i}}.$$    
By Lemma~\ref{lem:valery},
we may assume without loss of generality that $b \leq b_{i} \leq a_{i} \leq a$.
We now use Lemma~\ref{lem:garl} and Remark~\ref{rem:lepen}
to deduce that 
$$V_{\alpha (a) ; \alpha (b)}   
=
\bigcup_{i=1}^{m} V_{\alpha (a_{i}) ; \alpha (b_{i})}.$$

Finally, suppose that
$$\bigcup_{i=1}^{m} V_{a_{i}; b_{i}}
=
\bigcup_{j=1}^{n} V_{c_{j} ; d_{j}}$$
where $\{V_{a_{i};b_{i}} \colon 1 \leq i \leq m \}$ and $\{V_{c_{j};d_{j}} \colon 1 \leq j \leq n \}$
are compatible subsets of $\mathsf{B}(S)$.
Then 
$$V_{a_{i} ; b_{i}} \subseteq \bigcup_{j=1}^{n} V_{c_{j} ; d_{j}}$$ 
for each $i$.
Thus 
$$
V_{a_{i} ; b_{i}} 
= \bigcup_{j=1}^{n} V_{c_{j} ; d_{j}}V_{a_{i};b_{i}}^{-1}V_{a_{i};b_{i}}
= \bigcup_{j=1}^{n} V_{c_{j}a_{i}^{-1}a_{i} ; c_{j}a_{i}^{-1}b_{i} \vee c_{j}b_{i}^{-1}a_{i} \vee d_{j}a_{i}^{-1}a_{i}}.
$$
This makes essential use of the fact that we are working in an inverse semigroup.
Observe that
$$V_{c_{j}a_{i}^{-1}a_{i} ; c_{j}a_{i}^{-1}b_{i} \vee c_{j}b_{i}^{-1}a_{i} \vee d_{j}a_{i}^{-1}a_{i}}
\subseteq
V_{c_{j};d_{j}}.$$
Thus by the previous result and part (2) of Proposition~\ref{prop:gaulle},
we deduce that
$$V_{\alpha (a_{i}) ; \alpha (b_{i})} \subseteq  \bigcup_{j=1}^{n} V_{\alpha (c_{j}) ; \alpha (d_{j})}.$$   
Thus 
$$\bigcup_{i=1}^{m} V_{\alpha (a_{i}) ; \alpha (b_{i})}
\subseteq
\bigcup_{j=1}^{n} V_{\alpha (c_{j}) ; \alpha (d_{j})}.$$   
By symmetry we get equality.
\end{proof}

\begin{center}
{\bf We can now complete the proof of Theorem~\ref{them:bc}.}
\end{center}

\begin{proof}
Let $S$ be a distributive inverse semigroup and let $\theta \colon S \rightarrow T$ be any morphism to a Boolean inverse semigroup $T$.
Define $\gamma \colon \mathsf{B}(S) \rightarrow T$
by
$$\gamma \left(  \bigcup_{i=1}^{m} V_{a_{i} ; b_{i}} \right) 
= 
 \bigvee_{i=1}^{m} \alpha (a_{i}) \setminus \alpha (b_{i})$$
where $\{V_{a_{i};b_{i}} \colon 1 \leq i \leq m\}$ is a compatible subset of $\mathsf{B}(S)$.
This is a well-defined function by Proposition~\ref{prop:britannia} and Lemma~\ref{lem:properties}
and it is a homomorphism by Lemma~\ref{lem:benn} and a morphism by construction.
Clearly, $\theta = \beta \gamma$ and it is immediate that $\gamma$ is the unique morphism making the diagram of maps commute.
\end{proof}

\begin{proposition}\label{prop:mao} Let $S$ be an inverse semigroup.
Then $S$ is a weak semilattice if and only if $\mathsf{B}(S)$ is a $\wedge$-semigroup.
\end{proposition}
\begin{proof} By Lemma~\ref{lem:lenin}, we may assume that $S$ is distributive.
Thus we need to prove that $S$ is a $\wedge$-semigroup if and only if $\mathsf{B}(S)$ is a $\wedge$-semigroup.
Observe that to prove $\mathsf{B}(S)$ is a $\wedge$-semigroup, we need only consider intersections of the form $V_{a;b} \cap V_{c;d}$.
Suppose that $S$ is a $\wedge$-semigroup.
Then it is routine to check that 
$$V_{a;b} \cap V_{c;d} = V_{(a \wedge c);((b \wedge c) \vee (a \wedge d))}$$
and makes sense.
Conversely, suppose that $V_{a;b} \cap V_{c;d}$ is an element of $\mathsf{B}(S)$.
Then $V_{a;b} \cap V_{c;d} = \bigcup_{i=1}^{m} V_{x_{i};y_{i}}$.
By Lemma~\ref{lem:valery}, we may assume that 
$b \leq y_{i} \leq x_{i} \leq a$
and 
$d \leq y_{i} \leq x_{i} \leq c$
for $1 \leq i \leq m$.
Let $a,b \in S$.
Then by the above
$V_{a} \cap V_{b} = \bigcup_{i=1}^{m} V_{x_{i};y_{i}}$.
Put $c = \bigvee_{i=1}^{m}x_{i}$, well-defined since $x_{i} \leq a,b$ and so $\{x_{1}, \ldots, x_{n}\}$ is a compatible subset.
We claim that $V_{a} \cap V_{b} = V_{c}$.
To prove this, let $a,b \in P$, a prime filter. 
Then $x_{i} \in P$ for some $i$.
Thus $c \in P$.
Conversely, suppose that $c \in P$, a prime filter.
Then $x_{i} \in P$ for some $i$.
Thus $a,b \in P$.
We finish off by proving that $a \wedge b$ exists and equals $c$.
Clearly, $c \leq a,b$.
Let $d \leq a,b$.
Then $V_{d} \subseteq V_{a},V_{b}$.
Thus $V_{d} \subseteq V_{c}$.
It follows by part (2) of Proposition~\ref{prop:key} that $d \leq c$.
We have therefore proved that all binary meets exist.
\end{proof}

%%%%%%%%%%%%%%%%%%%%%%%%%%%%%%%%%%%%%%%%%%%%%%%%%%%%%%%%%%%%%%%%%%%%%%%%%%%%%%%%%%%%%%%%%%%%%
\subsection{Paterson's universal groupoid}

In this section, we shall complete the description of the connection between the Booleanization of an inverse semigroup 
and Paterson's universal groupoid first discussed in \cite[Section 5.1]{LL}.
Paterson described his universal groupoid in \cite[Sections~4.3, 4,4]{Paterson},
the significant part of this being \cite[Proposition 4.4.2]{Paterson}
where he proves that the inverse semigroup he denotes by $S''$ is isomorphic to the Boolean inverse semigroup associated with his universal groupoid.
There are two steps in constructing $S''$.
The first is the construction of what he calls $S'$ \cite[Page 176]{Paterson}
which corresponds to the inverse semigroup $\mathsf{V}(S)$ defined prior to Lemma~\ref{lem:fillon};
our semigroup appears simpler than the one described by Paterson because we are able to assume that our source inverse semigroup is distributive.
The second is the construction of the semigroup $S''$, defined in \cite[page~190]{Paterson}, which is the semigroup $\mathsf{V}(S)^{\vee} = \mathsf{B}(S)$.
Observe that the construction of Paterson's groupoid has been increasingly viewed from a more algebraic perspective \cite{Kellendonk1, Kellendonk2, Lenz, LMS}.

Let $S$ be an inverse semigroup with zero.
We know that $\mathsf{B}(S) \cong \mathsf{B}(\mathsf{D}(S))$.
The first step, therefore, is to relate the prime filters in $\mathsf{D}(S)$ to, what will turn out to be, the filters in $S$.

\begin{lemma}\label{lem:stasi} Let $S$ be an inverse semigroup with zero.
Then there is an order isomorphism between the proper filters of $S$ and the prime filters of $\mathsf{D}(S)$.
\end{lemma}
\begin{proof}
Let $F$ be a proper filter in $S$.
Define
$$F^{u} = \{ A \in \mathsf{D}(S) \colon A \cap F \neq \varnothing \}.$$
We prove that $F^{u}$ is a prime filter in $\mathsf{D}(S)$.
Let $A,B \in F^{u}$.
Then $a \in A \cap F$ and $b \in B \cap F$ for some $a$ and some $b$.
In particular, $a,b \in F$ and so, by assumption, there is a non-zero element $c \in F$ such that
$c \leq a,b$.
Now $c^{\downarrow} \in \mathsf{D}(S)$ and $c^{\downarrow} \subseteq A,B$,
since both $A$ and $B$ are order ideals.
It is clear that $F^{u}$ is closed upwards.
Thus $F^{u}$ is a proper filter (proper since it cannot contain $0^{\downarrow}$).
It remains to show that it is prime.
Let $A \cup B \in F^{u}$, where $A$ and $B$ are compatible elements of $\mathsf{D}(S)$.
Then $F \cap (A \cup B) \neq \varnothing$.
It follows that $F \cap A \neq \varnothing$ or $B \cap B \neq \varnothing$ and so
$A \in F^{u}$ or $B \in F^{u}$, as required.

Let $F_{1}$ and $F_{2}$ be proper filters of $S$ such that $F_{1} \subseteq F_{2}$.
Then it is immediate that $F_{1}^{u} \subseteq F_{2}^{u}$.

Let $P$ be a prime filter in $\mathsf{D}(S)$.
Define
$$P^{d} = \{s \in S \colon s^{\downarrow} \in P\}.$$
We prove that $P^{d}$ is a proper filter in $S$.
Observe first that since $P$ is a prime filter $0^{\downarrow} \notin P$.
Thus $0 \notin P^{d}$.
Let $s,t \in P^{d}$.
Then $s^{\downarrow},t^{\downarrow} \in P$.
But $P$ is a prime filter and so there is an element $A \in P$ such that $A \subseteq s^{\downarrow},t^{\downarrow}$.
Now if $A = \{a_{1},\ldots, a_{m}\}^{\downarrow}$ then $A = \bigvee_{i=1}^{m} a_{i}^{\downarrow}$.
But $P$ is a prime filter and so $a_{i}^{\downarrow} \in P$ for some $i$.
But then $a_{i} \in P^{d}$ and $a_{i} \leq s,t$.
Let $s \in P^{d}$ and $s \leq t$.
Then $s^{\downarrow} \in P$ but $s^{\downarrow} \subseteq t^{\downarrow}$ and so $t^{\downarrow} \in P$ giving $t \in P^{d}$.
We have therefore proved that $P^{d}$ is a proper filter.

Let $P_{1}$ and $P_{2}$ be prime filters in  $\mathsf{D}(S)$ such that $P_{1} \subseteq P_{2}$.
Then it is immediate that $P_{1}^{d} \subseteq P_{2}^{d}$.

It remains to iterate these two constructions.
Let $F$ be a proper filter in $S$.
We prove that $(F^{u})^{d} = F$.
Observe that
$$s \in (F^{u})^{d} \Leftrightarrow s^{\downarrow} \in F^{u} \Leftrightarrow s^{\downarrow} \cap F \neq \varnothing \Leftrightarrow s \in F,$$ 
where we use the fact that $F$ is a filter.
Let $P$ be a prime filter in $\mathsf{D}(S)$.
We prove that $(P^{d})^{u} = P$.
Observe that
$$A \in (P^{d})^{u} \Leftrightarrow A \cap P^{d} \neq \varnothing \Leftrightarrow s^{\downarrow} \subseteq A \mbox{ for some } s^{\downarrow} \in P
\Leftrightarrow A \in P.$$
\end{proof}

%%%%%%%%%%%%%%%%%%%%%%%%%%%%%%%%%%%%%%%%%%%%%%%%%%%%%%%%%%%%%%%%%%%%%%%%%%%%%%%%%%%%%%%%%%%%%%%%%%%%%%%%%%%%%%%%%%%%%%%%%%%%%%%%%%%%%%%%%%%%%%%%%%%%%%%%%%%%%%%%%%%%%%%%%%%%%%
Let $S$ be an inverse semigroup with zero.
Denote by $\mathcal{L}(S)$ the set of all proper filters of $S$.
This set becomes a groupoid in the following way.
If $A$ is a proper filter define $\mathbf{d}(A) = (A^{-1}A)^{\uparrow}$ and $\mathbf{r}(A) = (AA^{-1})^{\uparrow}$;
both $\mathbf{d}(A)$ and $\mathbf{r}(A)$ are proper filters.
Define $A \cdot B = (AB)^{\uparrow}$ if and only if $\mathbf{d}(A) = \mathbf{r}(A)$.
Then with this partial binary operation, the set $\mathcal{L}(S)$ is a groupoid.
The identities are precisely the proper filters that contain idempotents.
The order isomorphism of Lemma~\ref{lem:stasi} can be extended to an isomorphism of groupoids.

\begin{lemma}\label{lem:kgb} Let $S$ be an inverse semigroup with zero.
Then $F \mapsto F^{u}$ defines a functor from $\mathcal{L}(S)$ to $\mathsf{G}(\mathsf{D}(S))$,
$P \mapsto P^{d}$ defines a functor from $\mathsf{G}(\mathsf{D}(S))$ to $\mathcal{L}(S)$
and these functors are mutually inverse.
\end{lemma}
\begin{proof} We prove that $F \mapsto F^{u}$ is a functor.
If $F$ is an identity filter then every element of $F$ lies above an idempotent of $F$.
We prove that $F^{u}$ contains an identity which is enough to show that $F^{u}$ is an identity.
Let $A \in F^{u}$.
Then $A \cap F \neq \varnothing$.
Then there exists $a \in A$ and $a \in F$.
But $A$ is an order ideal and $a$ is above an idempotent, $e$ say, in $F$.
Thus $e \in A$ and $e \in F$.
It follows that $A \in F^{u}$ contains an idempotent and so is an identity.
\end{proof}

%%%%%%%%%%%%%%%%%%%%%%%%%%%%%%%%%%%%%%%%%%%%%%%%%%%%%%%%%%%%%%%%%%%%%%%%%%%%%%%%%%%%%%%%%%%%%%%%%%%%%%%%%%%%%%%%%%%%%%%%%%%%%%%%%%%%%%%%%%%
Let $S$ be an inverse semigroup with zero.
For each $a \in S$ define $U_{a}$ to be the set of all proper filters that contain $a$.
Put $\sigma = \{U_{a} \colon a \in S \}$.

\begin{lemma}\label{lem:hill} Let $S$ be an inverse semigroup.
\begin{enumerate}
\item $U_{0} = \varnothing$. 
\item $U_{a} = U_{b}$ if and only if $a = b$.
\item $U_{a}^{-1} = U_{a^{-1}}$.
\item $U_{a}U_{b} = U_{ab}$.
\item $U_{a}$ is a partial bisection.
\item $U_{a} \cap U_{b} = \bigcup_{x \leq a,b} U_{x}$.
\end{enumerate}
\end{lemma}
\begin{proof} (1) Immediate.
(2) This follows by the fact that $a^{\uparrow} \in U_{a}$.
(3) Straightfoward.
(4) Let $a \in A$ and $b \in B$, where $A$ and $B$ are filters.
Then from the definition of $A \cdot B$ we have that $ab \in A \cdot B$.
It follows that $U_{a}U_{b} \subseteq U_{ab}$.
To prove the reverse inclusion, let $ab \in C$, a filter.
Since $\mathbf{d}(ab) \leq \mathbf{d}(b)$ and so $B = (b\mathbf{d}(C))^{\uparrow}$ is a well-defined filter.
Since $b\mathbf{d}(ab) \in B$ we have that $\mathbf{d}(a)\mathbf{r}(b) \in \mathbf{r}(B)$.
It follows that $\mathbf{d}(a) \in \mathbf{r}(B)$ and so $A = (a\mathbf{r}(B))^{\uparrow}$ is a well-defined filter.
Clearly, $A \in U_{a}$, $B \in U_{b}$ and $C = A \cdot B$.
(5) This is immediate from the standard properties of filters.
(6) Straightfoward.\end{proof}

We therefore have an injective homomorphism $\upsilon \colon S \rightarrow \mathsf{L}(\mathcal{L}(S))$.
The identities of $\mathcal{L}(S)$ are the filters that contain idempotents (which is equivalent to saying that the filter is an inverse subsemigroup).
Let $a \in S$ and $a_{1}, \ldots, a_{m}$.
Define 
$$U_{a;a_{1}, \ldots, a_{m}} = U_{a} \cap U_{a_{1}}^{c} \cap \ldots \cap U_{a_{m}}^{c}.$$
Clearly, $U_{a;a_{1}, \ldots, a_{m}}$ is a partial bisection and so an element of $\mathsf{L}(\mathcal{L}(S))$.
The proof of the following is immediate.

\begin{lemma} Let $S$ be an inverse semigroup.
Then 
$$U_{a;a_{1}, \ldots, a_{m}} = U_{a;a_{1}} \cap \ldots \cap U_{a;a_{m}}$$
where the intersection is a compatible meet in the inverse semigroup since $U_{a;a_{i}} \subseteq U_{a}$.
\end{lemma}

Define $\Omega$ to be the set of all sets of the form $U_{a;a_{1}, \ldots, a_{m}}$.
It is easy to check that $\sigma$ is a basis for a topology on $\mathcal{L}(S)$.
In fact, using arguments similar to those in \cite[Lemma~2.7, Proposition~2.8]{LL}, it follows that $\mathcal{L}(S)$ is an \'etale topological groupoid.

If $S$ is now a distributive inverse semigroup then the set $\pi = \{V_{a} \colon a \in S\}$ forms the basis of a topology for $\mathsf{G}(S)$
which makes it an \'etale topological groupoid.
The functor of Lemma~\ref{lem:kgb} is actually continuous
and so links the filter topology constructed from $S$ with the prime filter topology constructed from $\mathsf{D}(S)$.

\begin{proposition}\label{prop:tosh} Let $S$ be an inverse semigroup with zero.
Then the \'etale topological groupoids $\mathcal{L}(S)$ and $\mathsf{G}(\mathsf{D}(S))$ are isomorphic.
\end{proposition}
\begin{proof} The basic open set $V_{s}$ is mapped to the basic open set $V_{s^{\downarrow}}$ 
and the inverse image of $V_{s^{\downarrow}}$ is $V_{s}$.
\end{proof}

Recall that a {\em (locally compact) Boolean space} is a $0$-dimensional, locally compact Hausdorff space.
An \'etale topological groupoid is said to be {\em Boolean} if its identity space is a Boolean space.
The topologies on the groupoids  $\mathcal{L}(S)$ and $\mathsf{G}(\mathsf{D}(S))$ now have to be refined in order that both groupoids become Boolean groupoids.

Let $\mathbf{s},\mathbf{t} \in \mathsf{D}(S)$ where $\mathbf{t} \leq \mathbf{s}$ in $\mathsf{D}(S)$.
We shall describe the set $V_{\mathbf{s};\mathbf{t}}$ explicitly.
Suppose that $\mathbf{s} = \{s_{1}, \ldots, s_{m} \}^{\downarrow}$ and $\mathbf{t} = \{t_{1},\ldots, t_{n} \}^{\downarrow}$.
In $\mathsf{D}(S)$, we have that $\mathbf{s} = \bigvee_{i=1}^{m} s_{i}^{\downarrow}$
and $\mathbf{t} = \bigvee_{j=1}^{n} t_{j}^{\downarrow}$.
By repeated application of part (3) of Lemma~\ref{lem:benn},
we may restrict out attention to the sets of the form $V_{s^{\downarrow};t_{1}^{\downarrow} \vee \ldots \vee t_{n}^{\downarrow}}$.

We now describe Paterson's universal groupoid\footnote{We have modified Paterson's construction in the obvious way to deal with the case where the inverse semigroup has a zero.} 
$\mathsf{G}_{u}(S)$ of the inverse semigroup $S$.
The underlying groupoid is still $\mathcal{L}(S)$ but a different topology is defined using $\Omega$ as a basis.
With respect to the topology with basis $\Omega$,  
the groupoid is called the {\em universal groupoid} and is denoted by means of $\mathsf{G}_{u}(S)$ where `u' stands for `universal'.

Define the topological groupoid  $\mathsf{G}(\mathsf{D}(S))^{\dagger}$ to have the same underlying groupoid as  $\mathsf{G}(\mathsf{D}(S))$
but with the topology having the basis the sets $V_{s;t}$. We call this the {\em patch topology}.\footnote{Compare with \cite{J}.}

\begin{proposition}\label{prop:soggy} Let $S$ be an inverse semigroup with zero.
The universal groupoid $\mathsf{G}_{u}(S)$ is homeomorphic to $\mathsf{G}(\mathsf{D}(S))^{\dagger}$.
\end{proposition}
\begin{proof} We shall use the same groupoid isomorphism as in Proposition~\ref{prop:tosh} derived from Lemma~\ref{lem:kgb}.
Observe that the basic open set $U_{s;s_{1}, \ldots, s_{m}}$ is mapped to the basic open set $V_{s^{\downarrow};\{s_{1}, \ldots, s_{m}\}^{\downarrow}}$
and that the inverse image of $V_{s^{\downarrow};\{s_{1}, \ldots, s_{m}\}^{\downarrow}}$ is $U_{s;s_{1}, \ldots, s_{m}}$. 
\end{proof}

At this point, we apply non-commutative Stone duality \cite{LL}.
If $G$ is a Boolean groupoid then the set of all compact-open local bisections of $G$, denoted by $\mathsf{KB}(G)$, is a Boolean inverse semigroup.
The following theorem establishes the exact connection between the Booleanization of an inverse semigroup and its associated universal groupoid.
The bridge between the two being provided by Proposition~\ref{prop:soggy}.

\begin{theorem}\label{them:paterson} Let $S$ be an inverse semigroup with zero.
Then $\mathsf{B}(S) \cong \mathsf{KB}(\mathsf{G}_{u}(S))$.
\end{theorem}

%%%%%%%%%%%%%%%%%%%%%%%%%%%%%%%%%%%%%%%%%%%%%%%%%%%%%%%%%%%%%%%%%%%%%%%%%%%%%%%%%%%%%%%%%%%%%%%%%%%%%%%%%%%%%%%%%%%%%%%%%%%%%%%%%%%%%%%%%%%%%%%%%%%%%%%%%%%%%%%
\subsection{A direct construction of the Booleanization}

Our goal now is to give a direct construction of the Booleanization of an inverse semigroup with zero $S$.
Of course, our preceding calculations do just this but mediated by the distributive completion of $S$.
In what follows, we shall not use this completion at all, but instead use the constructions described in \cite{Lenz,LMS}. 
The following will be used repeatedly: if $a \neq 0$ then $a^{\uparrow} \in U_{a}$.

\begin{lemma}\label{lem:house} Let $S$ be an inverse semigroup with zero.
\begin{enumerate}

\item Suppose that $a \neq 0$. Then $U_{a;a_{1}, \ldots, a_{m}} = \varnothing$ if and only if $a = a_{i}$ for some $1 \leq i \leq m$,

\item $U_{a;a_{1}, \ldots, a_{m}}^{-1} = U_{a^{-1};a_{1}^{-1}, \ldots, a_{m}^{-1}}$.

\item $U_{a;a_{1}, \ldots, a_{m}} U_{b;b_{1}, \ldots, b_{m}} = U_{ab; ab_{1}, \ldots, ab_{n}, a_{1}b, \ldots, a_{m}b}$.

\item Suppose that $U_{a;a_{1}, \ldots, a_{m}} \neq \varnothing$.
Then $U_{a;a_{1}, \ldots, a_{m}}^{2} = U_{a;a_{1}, \ldots, a_{m}}$ if and only if $a$ is an idempotent.

\item Let $e$ and $f$ be idempotents and suppose that $e_{1}, \ldots, e_{m} < e$ and $f_{1}, \ldots, f_{n} < f$.
Then $U_{e;e_{1}, \ldots, e_{m}} \cap U_{f;f_{1}, \ldots, f_{n}} = U_{ef;fe_{1}, \ldots, fe_{m}, ef_{1}, \ldots, ef_{n}}$.

\item $U_{a;a_{1}, \ldots, a_{m}}^{-1}U_{a;a_{1}, \ldots, a_{m}} = U_{\mathbf{d}(a);\mathbf{d}(a_{1}), \ldots, \mathbf{d}(a_{m})}$.

\end{enumerate}
\end{lemma}
\begin{proof} (1) Observe that $a^{\uparrow} \in U_{a}$.
If  $U_{a;a_{1}, \ldots, a_{m}} = \varnothing$ then $a_{i} \in a^{\uparrow}$ for some $1 \leq i \leq m$.
This means precisely that $a = a_{i}$.
The proof of (2) is straightforward. We prove (3).
We show first that 
$U_{a;a_{1}, \ldots, a_{m}}U_{b;b_{1}, \ldots, b_{n}} \subseteq U_{ab; ab_{1}, \ldots, ab_{n}, a_{1}b, \ldots, a_{m}b}$.
Clearly, the lefthand side is contained in $U_{ab}$.
Let $A \in U_{a;a_{1}, \ldots, a_{m}}$ and $B \in  U_{b;b_{1}, \ldots, b_{n}}$ where $A \cdot B$ is defined.
Suppose that $ab_{i} \in A \cdot B$.
Then $a'b' \leq ab_{i}$ where $a' \in A$ and $b' \in B$.
Thus $a^{-1}a'b' \leq b_{i}$.
It follows from the fact that $A \cdot B$ is defined that $b_{i} \in B$, which is a contradiction.
Similarly, if $a_{j}b \in A \cdot B$ then $a_{j} \in A$, which is a contradiction.
We have therefore proved that the left-hand side is contained in the right-hand side.
To prove the reverse inclusion, let $C \in U_{ab; ab_{1}, \ldots, ab_{n}, a_{1}b, \ldots, a_{m}b}$.
Then by Lemma~\ref{lem:hill}, there exists $A \in U_{a}$ and $B \in U_{b}$ such that $A \cdot B = C$.
If $a_{j} \in A$ and given that $b \in B$ it follows that $a_{j}b \in C$, which is a contradiction.
Similarly, if $b_{i} \in B$ then $ab_{i} \in C$, which is a contradiction.
The result now follows.
(4) By assumption, $a_{1}, \ldots, a_{m} < a$.
By (1), it follows that $U_{a;a_{1}, \ldots, a_{m}} \neq \varnothing$.
Suppose that $U_{a;a_{1}, \ldots, a_{m}}^{2} = U_{a;a_{1}, \ldots, a_{m}}$.
Then $U_{a^{2};aa_{1}, \ldots, aa_{m},a_{1}a, \ldots, a_{m}a} = U_{a;a_{1}, \ldots, a_{m}}$ by (3).
Observe that $U_{a^{2};aa_{1}, \ldots, aa_{m},a_{1}a, \ldots, a_{m}a} \neq \varnothing$.
It follows that $aa_{1}, \ldots, aa_{m},a_{1}a, \ldots, a_{m}a < a^{2}$.
Thus the proper filters $a^{\uparrow}$ and $(a^{2})^{\uparrow}$ belong to both sides.
This implies that $a = a^{2}$, as claimed.
Suppose now that $a$ is an idempotent.
By (3), we have that $U_{a;a_{1}, \ldots, a_{m}}^{2} = U_{a;aa_{1}, \ldots, aa_{m},a_{1}a, \ldots, a_{m}a}$.
Since $a$ is an idempotent so too are $a_{1}, \ldots, a_{m}$.
In addition, since $a_{i} \leq a$ it follows that $aa_{i} = a_{i}a = a_{i}$.
It follows that $U_{a;aa_{1}, \ldots, aa_{m},a_{1}a, \ldots, a_{m}a} = U_{a;a_{1}, \ldots, a_{m}}$, as claimed.
(5) There are two cases to consider: $ef = 0$ and $ef \neq 0$.
If $ef = 0$ then the intersection is the empty set.
We therefore assume in what follows that $ef \neq 0$.
Suppose that $ef = fe_{i}$.
Then the right-hand side is empty.
Let $A$ be a proper filter in $U_{e;e_{1}, \ldots, e_{m}} \cap U_{f;f_{1}, \ldots, f_{n}}$.
Then $ef \in A$ and so $fe_{i} \in A$ yielding $e_{i} \in A$ which is a contradiction.
It follows that the left-hand side is empty as well.
A similar results follows if $ef = ef_{j}$.
We may therefore assume that $fe_{i}, ef_{j} < ef$.
Suppose that $A \in U_{e;e_{1}, \ldots, e_{m}} \cap U_{f;f_{1}, \ldots, f_{n}}$.
Then $e,f \in A$ and so $ef \in A$.
It is immediate that $fe_{i}, ef_{j} \notin A$.
Suppose that $A \in U_{ef;fe_{1}, \ldots, fe_{m}, ef_{1}, \ldots, ef_{n}}$.
Then $e,f \in A$ and it is clear that $e_{i}, f_{j} \notin A$.
(6) This is a routine calculation based on parts (2) and (3).
\end{proof}

Recall that $\mathcal{L}(S)$ is the groupoid of proper filters on $S$ and that $\mathsf{L}(\mathcal{L}(S))$
is the Boolean inverse semigroup of all partial bisections of $\mathcal{L}(S)$.
Define $\mathcal{B}(S)$ to be the subset of $\mathsf{L}(\mathcal{L}(S))$
that consists of finite compatible joins of elements of the form $U_{a:a_{1}, \ldots, a_{m}}$.
By Lemma~\ref{lem:house}, this is an inverse semigroup and by Lemma~\ref{lem:hill}, 
there is an embedding  $\upsilon \colon S \rightarrow \mathcal{B}(S)$ given by $a \mapsto U_{a}$;
it is clear that $\mathcal{B}(S)$ is a distributive inverse semigroup, that it is actually Boolean
follows from part (5) of Lemma~\ref{lem:house} and, say, \cite[Lemma~2.6]{Spielberg2014}.
We shall now prove directly that the map $\upsilon \colon S \rightarrow \mathcal{B}(S)$ is the Booleanization of $S$.
We will need the following sequence of lemmas.

\begin{lemma}\label{lem:rain} Let $S$ be an inverse semigroup with zero and let $\alpha \colon S \rightarrow T$ be a homomorphism
to a Boolean inverse semigroup.
  If 
$$U_{a; a_{1}, \ldots, a_{m}} = U_{b; b_{1}, \ldots, b_{n}}$$ 
then
$$\alpha (a) \setminus (\alpha (a_{1}) \vee \ldots \vee \alpha (a_{m}))
=
\alpha (b) \setminus (\alpha (b_{1}) \vee \ldots \vee \alpha (b_{n})).$$
\end{lemma}
\begin{proof} Without loss of generality, we can assume that 
$U_{a; a_{1}, \ldots, a_{m}} = U_{b; b_{1}, \ldots, b_{n}} \neq \varnothing$.
By part (1) of Lemma~\ref{lem:house}, it follows that $a_{1}, \ldots, a_{m} < a$ and $b_{1}, \ldots, b_{n} < b$.
By using the proper filters $a^{\uparrow}$ and $b^{\uparrow}$, we easily deduce that $a = b$.
We shall now prove that $\alpha (a_{1}) \vee \ldots \vee \alpha (a_{m}) = \alpha (b_{1}) \vee \ldots \vee \alpha (b_{n})$.
Suppose they are not equal.
Then, without loss of generality, there is a prime filter $P$ in $T$ such that
$\alpha (a_{1}) \vee \ldots \vee \alpha (a_{m}) \in P$ and $\alpha (b_{1}) \vee \ldots \vee \alpha (b_{n}) \notin P$.
Relabelling if necessary, we may suppose that $\alpha (a_{1}) \in P$ and $\alpha (b_{1}), \ldots, \alpha (b_{n}) \notin P$.
Clearly, $\alpha^{-1}(P) \neq \varnothing$ since $a_{1} \in \alpha^{-1}(P)$.
The set $\alpha^{-1}(P)$ is also closed upwards.
It follows that $a_{1}^{\uparrow} \subseteq \alpha^{-1}(P)$.
By construction, $b_{1}, \ldots, b_{n} \notin a_{1}^{\uparrow}$.
Thus $a_{1}^{\uparrow} \in U_{b; b_{1}, \ldots, b_{n}}$ but $a_{1}^{\uparrow} \notin U_{a; a_{1}, \ldots, a_{m}}$,
which is a contradiction.
\end{proof}

\begin{lemma}\label{lem:lilies} Let $S$ be an inverse semigroup with zero and let $\alpha \colon S \rightarrow T$ be a homomorphism
to a Boolean inverse semigroup.
If 
$$U = U_{a; a_{1}, \ldots, a_{m}} \sim V = U_{b; b_{1}, \ldots, b_{n}}$$ 
then
$$\alpha (a) \setminus (\alpha (a_{1}) \vee \ldots \vee \alpha (a_{m}))
\sim
\alpha (b) \setminus (\alpha (b_{1}) \vee \ldots \vee \alpha (b_{n})).$$
\end{lemma}
\begin{proof} By assumption, $U^{-1}V$ and $UV^{-1}$ are idempotents.
Suppose both are non-empty.
Then by part (4) of Lemma~\ref{lem:house}, both $a^{-1}b$ and $ab^{-1}$ are idempotents and so $a \sim b$.
It follows that $\alpha (a) \sim \alpha (b)$ from which the claim is seen to be true.
Suppose that $U^{-1}V$ is empty.
Then by part (1) of Lemma~\ref{lem:house}, either $a^{-1}b = a^{-1}b_{i}$, for some $i$, or $a^{-1}b = a_{j}^{-1}b$, for some $j$.
Suppose that $a^{-1}b = a^{-1}b_{i}$.
Then a straightforward calculation using Lemma~\ref{lem:properties} shows that 
$$(\alpha (a) \setminus (\alpha (a_{1}) \vee \ldots \vee \alpha (a_{m})))^{-1}
(\alpha (b) \setminus (\alpha (b_{1}) \vee \ldots \vee \alpha (b_{n}))) = 0.$$
The result now follows by symmetry.
\end{proof}

\begin{lemma}\label{lem:rain} Let $S$ be an inverse semigroup with zero and let $\alpha \colon S \rightarrow T$ be a homomorphism
to a Boolean inverse semigroup.
If 
$$U_{a; a_{1}, \ldots, a_{m}} \subseteq U_{b; b_{1}, \ldots, b_{n}}$$ 
then
$$\alpha (a) \setminus (\alpha (a_{1}) \vee \ldots \vee \alpha (a_{m}))
\leq
\alpha (b) \setminus (\alpha (b_{1}) \vee \ldots \vee \alpha (b_{n})).$$
\end{lemma}
\begin{proof} We are working in the inverse semigroup $\mathcal{B}(S)$.
Thus 
$$U_{a; a_{1}, \ldots, a_{m}} = U_{b; b_{1}, \ldots, b_{n}}U_{a; a_{1}, \ldots, a_{m}}^{-1}U_{a; a_{1}, \ldots, a_{m}}.$$
We therefore have that
$$U_{a;a_{1}, \ldots, a_{m}} = U_{b\mathbf{d}(a); b\mathbf{d}(a_{1}), \ldots, b\mathbf{d}(a_{m}), b_{1}\mathbf{d}(a), \ldots, b_{n}\mathbf{d}(a)}.$$
By Lemma~\ref{lem:rain}, we deduce that
$$\alpha (a) \setminus (\alpha (a_{1}) \vee \ldots \vee \alpha (a_{m}))
=
\alpha (b\mathbf{d}(a)) \setminus (\alpha (b\mathbf{d}(a_{1})) \vee \ldots \vee \alpha (b\mathbf{d}(a_{m})) \vee \alpha (b_{1}\mathbf{d}(a)) \vee \ldots \vee \alpha (b_{n}\mathbf{d}(a))).$$
Observe that
$$\alpha (b\mathbf{d}(a)) \setminus (\alpha (b\mathbf{d}(a_{1})) \vee \ldots \vee \alpha (b\mathbf{d}(a_{m})) \vee \alpha (b_{1}\mathbf{d}(a)) \vee \ldots \vee \alpha (b_{n}\mathbf{d}(a)))$$
is less than or equal to
$$\alpha (b) \setminus [\alpha (b)(\alpha (\mathbf{d}(a_{1})) \vee \ldots \vee \alpha (\mathbf{d}(a_{m})) \vee (\alpha (b_{1}) \vee \ldots \vee \alpha (b_{n}))]$$
by 
Lemma~\ref{lem:properties}.
This in turn is less than or equal to
$$\alpha (b) \setminus [\alpha (b_{1}) \vee \ldots \vee \alpha (b_{n})],$$
as required.
\end{proof}

%\begin{lemma}\label{lem:sleet} Let $S$ be an inverse semigroup with zero and let $\alpha \colon S \rightarrow T$ be a homomorphism
%to a Boolean inverse semigroup.
%If 
%$$U_{a; a_{1}, \ldots, a_{m}} \sim U_{b; b_{1}, \ldots, b_{n}}$$ 
%then
%$$\alpha (a) \setminus (\alpha (a_{1}) \vee \ldots \vee \alpha (a_{m}))
%\sim
%\alpha (b) \setminus (\alpha (b_{1}) \vee \ldots \vee \alpha (b_{n})).$$
%\end{lemma}
%\begin{proof}
%\end{proof}

\begin{lemma}\label{lem:hail} Let $S$ be an inverse semigroup with zero and let $\alpha \colon S \rightarrow T$ be a homomorphism
to a Boolean inverse semigroup.
If
$$U_{a;a_{1}, \ldots,a_{m}} = \bigcup_{i=1}^{m} U_{b_{i};b_{i1}, \ldots, b_{in_{i}}}$$
then
$$\alpha (a) \setminus (\alpha (a_{1}) \vee \ldots \vee \alpha (a_{m}))
=
\bigvee_{i=1}^{m} \alpha (b_{i}) \setminus (\alpha (b_{i1}) \vee \ldots \vee \alpha (b_{in_{i}})).$$
\end{lemma}
\begin{proof} Observe that  $U_{b_{i};b_{i1}, \ldots, b_{in_{i}}} \subseteq U_{a;a_{1}, \ldots,a_{m}}$,
this means that the  $U_{b_{i};b_{i1}, \ldots, b_{in_{i}}}$ are pairwise compatible,
and so by Lemma~\ref{lem:rain} and Lemma~\ref{lem:lilies}  
$\alpha (b_{i}) \setminus (\alpha (b_{i1}) \vee \ldots \vee \alpha (b_{in_{i}})) \leq \alpha (a) \setminus (\alpha (a_{1}) \vee \ldots \vee \alpha (a_{m}))$.
It follows that $\bigvee_{i=1}^{n} \alpha (b_{i}) \setminus (\alpha (b_{i1}) \vee \ldots \vee \alpha (b_{in_{i}})) \leq \alpha (a) \setminus (\alpha (a_{1}) \vee \ldots \vee \alpha (a_{m}))$.
To show that this inequality is, in fact, an equality let $P$ be a prime filter that contains
$\alpha (a) \setminus (\alpha (a_{1}) \vee \ldots \vee \alpha (a_{m}))$
and omits
$\bigvee_{i=1}^{n} \alpha (b_{i}) \setminus (\alpha (b_{i1}) \vee \ldots \vee \alpha (b_{in_{i}}))$.
Then $\alpha^{-1} (P)$ is an upwardly closed set that contains $a$ and omits $a_{1}, \ldots, a_{m}$.
It follows that $a^{\uparrow} \in U_{a;a_{1}, \ldots,a_{m}}$.
But then, for some $i$, we must have that $b_{i} \in a^{\uparrow}$ and $b_{i1}, \ldots, b_{in_{i}} \in P$,
which is a contradiction.
\end{proof}

\begin{lemma}\label{lem:snow} Let $S$ be an inverse semigroup with zero and let $\alpha \colon S \rightarrow T$ be a homomorphism
to a Boolean inverse semigroup.
If
$$U_{a;a_{1}, \ldots,a_{m}} \subseteq \bigcup_{i=1}^{m} U_{b_{i};b_{i1}, \ldots, b_{in_{i}}}$$
then
$$\alpha (a) \setminus (\alpha (a_{1}) \vee \ldots \vee \alpha (a_{m}))
\leq
\bigvee_{i=1}^{m} \alpha (b) \setminus (\alpha (b_{i1}) \vee \ldots \vee \alpha (b_{in_{i}})).$$
\end{lemma}
\begin{proof} By definition the $U_{b_{i};b_{i1}, \ldots, b_{in_{i}}}$ are pairwise compatible so we may apply Lemma~\ref{lem:lilies}.
We use the fact that we are working in the inverse semigroup $\mathcal{B}(S)$
which enables us to reduce to the case of Lemma~\ref{lem:hail}.
\end{proof}

\begin{lemma}\label{lem:frogs}
Let $S$ be an inverse semigroup with zero and let $\alpha \colon S \rightarrow T$ be a homomorphism
to a Boolean inverse semigroup.
If
$$\bigcup_{j=1}^{p}U_{a_{j};a_{j1}, \ldots,a_{jp_{j}}} = \bigcup_{i=1}^{s} U_{b_{i};b_{i1}, \ldots, b_{is_{i}}}$$
then
$$\bigvee_{j=1}^{p} \alpha (a_{j}) \setminus (\alpha (a_{j1}) \vee \ldots \vee \alpha (a_{jp_{j}}))
=
\bigvee_{i=1}^{s} \alpha (b_{i}) \setminus (\alpha (b_{i1}) \vee \ldots \vee \alpha (b_{is_{i}})).$$
\end{lemma}
\begin{proof} 
This is immediate by Lemma~\ref{lem:lilies} and Lemma~\ref{lem:snow}.
\end{proof}

We can now give the universal characterization of $\upsilon \colon S \rightarrow \mathcal{B}(S)$.
Let $\theta \colon S \rightarrow T$ be a homomorphism to a Boolean inverse semigroup.
Define $\Theta \colon \mathcal{B}(S) \rightarrow T$ by
$$\Theta \left( \bigcup_{i=1}^{s} U_{b_{i};b_{i1}, \ldots, b_{is_{i}}} \right) 
=
\bigvee_{i=1}^{s} \alpha (b_{i}) \setminus (\alpha (b_{i1}) \vee \ldots \vee \alpha (b_{is_{i}})).$$
By Lemma~\ref{lem:lilies} and Lemma~\ref{lem:frogs}, the map $\Theta$ is well-defined and $\upsilon \Theta = \theta$.
The fact that $\Theta$ is a homomorphism of inverse semigroups follows by part (3) of Lemma~\ref{lem:house} and 
part (4) of Lemma~\ref{lem:properties}.
It is then a morphism by construction.
It is clear that $\Theta$ is unique with these properties.

%%%%%%%%%%%%%%%%%%%%%%%%%%%%%%%%%%%%%%%%%%%%%%%%%%%%%%%%%%%%%%%%%%%%%%%%%%%%%%%%%%%%%%%%%%%%%%%%%%%%
\subsection{Computations}

In this section, we describe how to compute the Booleanization of a distributive inverse semigroup in a practical way.
Let $B$ be a Boolean inverse semigroup.
Let $D$ be an inverse subsemigroup of $B$ where $D$ is distributive in its own right.
We are interested in the way that $D$ sits inside $B$.
For clarity, denote the join in $D$ by $\vee$ and the join in $B$ by $\cup$.
Let $a,b \in D$ be compatible.
Thus both $a \vee b \in D$ and $a \cup b \in B$ exist.
By definition, $a \cup b \leq a \vee b$ but there is no reason for them to be equal.
For this to be the case, $(\mathsf{E}(D),\vee)$ must be a subalgebra of $(\mathsf{E}(B),\vee)$.
This leads us to the following definition.
Let $B$ be a Boolean inverse semigroup and let $D$ be an inverse subsemigroup of $B$ which is distributive.
We say that $D$ is a {\em distributive subalgebra} of $B$ if the distributive lattice $\mathsf{E}(D)$ is a subalgebra, with respect to meets, joins and bottom,
of the Boolean algebra  $\mathsf{E}(B)$.
In this case, joins in $D$ are identical to joins in $B$.
With meets, we have to be more careful but the only meets we shall be interested in are the compatible ones
which are constructed purely algebraically by Lemma~\ref{lem:macron}.
By Lemma~\ref{lem:properties} and Boolean algebra, we have the following.

\begin{lemma}\label{lem:paterson} Let $S$ be a Boolean inverse semigroup and let $D$ be an inverse subsemigroup distributive in its own right.
\begin{enumerate}
\item Put $D'$ equal to all elements of $S$ of the form $a \setminus b$ where $b \leq a$.
Then $D'$ is an inverse subsemigroup of $S$.
\item Put $D''$ equal to all joins of compatible, finite subsets of $D'$.
Then $D''$ is a Boolean inverse semigroup.
\end{enumerate}
\end{lemma}

Our goal now is to determine what conditions need to be imposed to ensure that $D''$ is in fact equal to $\mathsf{B}(D)$.
Let $D$ be a distributive subalgebra of the Boolean inverse semigroup $S$.
Put $\mathsf{B}_{S}(D) = D''$ using the above notation.
We call this the {\em Boolean hull} of $D$ in $S$.

\begin{theorem}\label{them:freedom} 
Let $D$ be a distributive subalgebra of the Boolean inverse semigroup $S$.
Then the Boolean hull of $D$ in $S$ is isomorphic to the Booleanization of $D$.
\end{theorem}
\begin{proof} By Theorem~\ref{them:bc} and its proof, there is a morphism $\gamma \colon \mathsf{B}(S) \rightarrow \mathsf{B}_{S}(D)$
given by
$$\gamma \left(  \bigcup_{i=1}^{m} V_{a_{i} ; b_{i}} \right) 
= 
 \bigvee_{i=1}^{m} a_{i} \setminus b_{i}$$
where $\{V_{a_{i};b_{i}} \colon 1 \leq i \leq m\}$ is a compatible subset of $\mathsf{B}(S)$.
From the construction of $\mathsf{B}_{S}(D)$ this morphism is surjective and so it just remains to prove that it is injective to prove the theorem.
The crux of the proof is to show that if
$$a \setminus b \leq  \bigvee_{i=1}^{m} a_{i} \setminus b_{i}$$
then
$$V_{a;b} \subseteq  \bigcup_{i=1}^{m} V_{a_{i} ; b_{i}}.$$ 
The important point to remember is that the first inequality holds in $S$ whereas the second in $D$.
Observe that $\mathbf{d}(a \setminus b) = \mathbf{d}(a) \setminus \mathbf{d}(b)$.
Now
$$(a_{i} \setminus b_{i})(\mathbf{d}(a) \setminus \mathbf{d}(b))
=
(a_{i}\mathbf{d}(a)) \setminus (a_{i}\mathbf{d}(b) \vee b_{i}\mathbf{d}(a)).$$
Put $a_{i}' = a_{i} \mathbf{d}(a)$
and 
$b_{i}' = \mathbf{d}(b) \vee b_{i}\mathbf{d}(a)$.
Observe that by our assumption $a_{i}',b_{i}' \in D$.
Thus
$$a \setminus b =  \bigvee_{i=1}^{m} a_{i}' \setminus b_{i}'.$$
By Lemma~\ref{lem:properties} and Lemma~\ref{lem:valery}, we may assume that
$b \leq b_{i}' \leq a_{i}' \leq a$.
We need to be careful about notation in what follows.
Let $s \in S$. I shall write $V_{s}^{S}$ for the set of all prime filters in $S$ that contain $s$.
By Lemma~\ref{lem:properties}, we have that
$$V_{a;b}^{S} =  \bigcup_{i=1}^{m} V_{a'_{i};b'_{i}}^{S}.$$
By Lemma~\ref{lem:garl} this translates into results about joins and compatible meets for elements of $B$ and so are equal to
joins and compatible meets in $D$ since $D$ is a distributive subalgebra of $S$.
Thus applying Lemma~\ref{lem:garl} in the opposite direction gives us $V_{a;b} =  \bigcup_{i=1}^{m} V_{a'_{i};b'_{i}}$.
But $V_{a_{i}';b_{i}'} \subseteq V_{a_{i};b_{i}}$ and from this our result follows.
\end{proof}

%%%%%%%%%%%%%%%%%%%%%%%%%%%%%%%%%%%%%%%%%%%%%%%%%%%%%%%%%%%%%%%%%%%%%%%%%%%%%%%%%%%%%%%%%%%%%%%%%%%
\section{Applications and examples}

%%%%%%%%%%%%%%%%%%%%%%%%%%%%%%%%%%%%%%%%%%%%%%%%%%%%%%%%%%%%%%%%%%%%%%%%%%%%%%%%%%%%%%%%%%%%%%%%%%%%%%%%%%%%%%%%%%%%%%%%%%%%%%%%%%%%%%%%%%%%%%%
\subsection{Representations of inverse semigroups in rings}

Observe that in this section, we deal with monoids;
the extension to semigroups is straightforward.

Marshall H. Stone, a functional analyst, became interested in Boolean algebras
through his work on the spectral theory of symmetric operators which in turn led to an interest in algebras of commuting projections.
Such algebras are naturally Boolean algebras: in fact, Stone proved that Boolean algebras and Boolean rings\footnote{A Boolean ring is a ring in which every element is an idempotent.
A simple exercise shows that such rings are always commutative.}
were two different ways of viewing the same class of structures \cite{Stone1936}.
Slightly more generally, Foster \cite{Foster} proved that the set of idempotents of any commutative ring was a Boolean algebra
when the following definitions were made: $e \vee f = e + f - ef$, $e \wedge f = e \cdot f$ and $e' = 1 - e$.
In this section, we shall be interested in inverse semigroups as subsemigroups of the multiplicative monoids of rings;
in particular, inverse semigroups as subsemigroups of the multiplicative monoids-with-involution of $C^{\ast}$-algebras.
We begin with a simple lemma.

\begin{lemma}\label{lem:dingbat} Let $S$ be an inverse submonoid (with zero) of the multiplicative monoid of a ring $R$.
Suppose that the following two conditions hold:
\begin{enumerate}
\item If $a,b \in S$ are orthogonal then $a + b \in S$.
\item If $e \in S$ is an idempotent then $1 - e \in S$.
\end{enumerate}
Then $S$ is a Boolean inverse monoid.
\end{lemma}
\begin{proof} Let $e$ and $f$ be orthogonal idempotents in $S$.
We prove first that $e \vee f$ exists in $S$ and equals $e + f$.
Clearly, $e + f$ is an idempotent and belongs to $S$ by assumption.
Observe that $e(e + f) = e$ and $f(e + f) = f$.
Thus $e,f \leq e + f$.
Suppose that $e,f \leq i$, where $i$ is an idempotent in $S$.
Then $i(e + f) = ie + if = e + f$.
Thus $e + f \leq i$.
We have therefore proved that $e \vee f = e + f$.
Now let $a$ and $b$ be orthogonal elements of $S$.
We prove that $a \vee b$ exists in $S$ and is equal to $a + b$. 
Put $c = a + b$.
Then $ca^{-1}a = a$ and $cb^{-1}b = b$.
Thus $a,b \leq c$.
But $\mathbf{d}(c) = \mathbf{d}(a) + \mathbf{d}(b) = \mathbf{d}(a) \vee \mathbf{d}(b)$.
It follows that $a \vee b = a + b$.
We have therefore shown that $S$ has all binary orthogonal joins and multiplication distributes over such joins.

We now prove that $\mathsf{E}(S)$ is a Boolean algebra.
Let $e,f \in \mathsf{E}(S)$.
Define $e \circ f = e + f - ef$ but $e + f - ef = e(1-f) + f$
and $e(1-f)$ and $f$ are orthogonal.
It follows that $\mathsf{E}(S)$ is closed under the binary operation $\circ$.
Let $e$ and $f$ be arbitrary idempotents.
We prove that $e \vee f$ exists and equals $e \circ f$.
Observe that $e(e \circ f) = e$ and $f(e \circ f) = f$ so that $e,f \leq e \circ f$.
Now let $e,f \leq i$.
It is easy to check that $i(e \circ f) = e \circ f$.
Thus $e \circ f \leq i$.
We have therefore proved that $e \vee f = e \circ f$.
It is clear that $e(i \vee j) = ei \vee ej$
and it is easy to show that $e \vee (ij) = (e \vee i)(e \vee j)$.
It is now routine to prove that $\mathsf{E}(S)$ is a Boolean algebra.

The lemma now follows by an application of Proposition~\ref{prop:definition}. 
\end{proof}

Our first main result is a slight generalization of a construction to be found in \cite[pp~175--176, pp~190--193]{Paterson}
although our proof is completely algebraic and there is no appeal to \cite{Wordingham}.
In the proof below, the construction of $S'$ deals with part (2) of Lemma~\ref{lem:dingbat}
and that of $S''$ deals with part (1) of  Lemma~\ref{lem:dingbat}.

\begin{proposition}\label{prop:twix} Let $S$ be an inverse submonoid (with zero) of the multiplicative monoid of a ring $R$.
Then there is a Boolean inverse submonoid $S''$ such that $S \subseteq S'' \subseteq R$.
\end{proposition}
\begin{proof} Observe first that if $e$ is an idempotent then $1 - e$ is an idempotent and if $ef = fe$ then $e(1-f) = (1-f)e$.
Define
$$E' = \{e(1 - e_{1}) \ldots (1-e_{n}) \colon e, e_{1}, \ldots, e_{n} \in \mathsf{E}(S)\} \cup \mathsf{E}(S).$$
Then $E'$ is a commutative idempotent subsemigroup of $R$ containing $\mathsf{E}(S)$.
In addition, $E'$ is closed under conjugation by elements of $S$.
To prove this,
let 
$$\mathbf{e} = e(1-e_{1}) \ldots (1 - e_{m})$$ 
and  $s \in S$.
Then 
$$s^{-1} \mathbf{e} s = s^{-1}e(1-e_{1}) \ldots (1 - e_{m})s.$$
But
$$s^{-1}e(1-e_{1}) \ldots (1 - e_{m})s = s^{-1}et \cdot s^{-1}(1-e_{1})s \cdot \ldots \cdot s^{-1}(1-e_{m})s$$
whereas
$s^{-1}(1 - i)s = s^{-1}s - s^{-1}is = s^{-1}s(1 - s^{-1}is)$.
The claim now follows.
%%%%%%%%%%%%%%%%%%%%%%%%%%%%%%%%%%%%%%%%%%%%%%%%%%%%%%%%%%%%%%%%%%%%%%%%%%%%%%%%%%%%%%%%%%%%%%%%%%%%%%%%%%%%%%%%%%%%%%%%%%%%%%

Put $S' = SE'$.
Let $\mathbf{s} = s e(1 - e_{1}) \ldots (1 - e_{m})$.
Then $\mathbf{s} = s (s^{-1}se)(1 - e_{1}) \ldots (1 - e_{m})$.
Thus we may assume, whenever convenient, that $e \leq s^{-1}s$.
Next, $e(1 - e_{1}) = e - ee_{1} = e - e(ee_{1}) = e(1 - ee_{1})$.
It follows that we may also assume, whenever convenient, that $e_{1}, \ldots, e_{m} \leq e$. 
We prove that $S'$ is an inverse semigroup with semilattice of idempotents $E'$.

First, we prove closure under multiplication.
Let $\mathbf{s} = se(1-e_{1}) \ldots (1 - e_{m})$ and $\mathbf{t} = tf(1 - f_{1}) \ldots (1 - f_{n})$.
Then $\mathbf{s} \mathbf{t} =  se(1-e_{1}) \ldots (1 - e_{m})tf(1 - f_{1}) \ldots (1 - f_{n})$.
Write $t = tt^{-1}t$.
Then
$\mathbf{s} \mathbf{t} = st[t^{-1}e(1-e_{1}) \ldots (1 - e_{m})t]f(1 - f_{1}) \ldots (1 - f_{n})$.
But we proved above that $E'$ is closed under conjugation by elements of $S$.
It follows that $S'$ is closed under multiplication.

Let $\mathbf{s} = se(1-e_{1}) \ldots (1 - e_{m})$ 
and define $\mathbf{s}^{-1} = e(1-e_{1}) \ldots (1 - e_{m})s^{-1}$.
Then $\mathbf{s}^{-1} = e(1-e_{1}) \ldots (1 - e_{m})s^{-1} = s^{-1}[se(1-e_{1}) \ldots (1 - e_{m})s^{-1}]$
and we now use the fact that $E'$ is closed under conjugation by elements of $S$.
It follows that if $\mathbf{s} \in S'$ then $\mathbf{s}^{-1} \in S'$.
It is easy to check that
$\mathbf{s} = \mathbf{s} \mathbf{s}^{-1} \mathbf{s}$
and 
$\mathbf{s}^{-1} = \mathbf{s}^{-1} \mathbf{s} \mathbf{s}^{-1}$.

Thus $S'$ is a regular semigroup.

To prove that $S'$ is inverse it is enough to prove that $\mathsf{E}(S') = E'$.
Let  $\mathbf{s} = se(1-e_{1}) \ldots (1 - e_{m})$ and suppose that $\mathbf{s}^{2} = \mathbf{s}$.
As we indicated above, we may assume that $e \leq s^{-1}s$ and that $e_{1}, \ldots, e_{m} \leq e$.
We prove that $\mathbf{s} \in E'$.
By assumption,
$$se(1-e_{1}) \ldots (1 - e_{m}) 
=
se(1-e_{1}) \ldots (1 - e_{m})se(1-e_{1}) \ldots (1 - e_{m}).$$
Thus multiplying this equation on the left by $s^{-1}$ we obtain
$$e(1-e_{1}) \ldots (1 - e_{m}) 
=
e(1-e_{1}) \ldots (1 - e_{m})se(1-e_{1}) \ldots (1 - e_{m}).$$
Now write $s = (ss^{-1})s$ and move the $ss^{-1}$ to the front to get
$$e(1-e_{1}) \ldots (1 - e_{m}) 
=
s[s^{-1}e(1-e_{1}) \ldots (1 - e_{m})s] e(1-e_{1}) \ldots (1 - e_{m}).$$
It follows from this equation that
$$e(1-e_{1}) \ldots (1 - e_{m}) 
=
[s^{-1}e(1-e_{1}) \ldots (1 - e_{m})s] e(1-e_{1}) \ldots (1 - e_{m}).$$
Thus 
$$e(1-e_{1}) \ldots (1 - e_{m}) 
=
se(1-e_{1}) \ldots (1 - e_{m})
=
\mathbf{s}.$$

We have therefore proved that $S'$ is an inverse semigroup in its own right.

%%%%%%%%%%%%%%%%%%%%%%%%%%%%%%%%%%%%%%%%%%%%%%%%%%%%%%%%%%%%%%%%%%%%%%%%%%%%%%%%%%%%%%%%%%%%%%%%%%%%%%%%%%%%%%
Now define $S'' \subseteq R$ to be the set of all finite sums of orthogonal elements of $S'$.
If $\{a_{1}, \ldots, a_{m}\}$ and $\{b_{1}, \ldots, b_{n}\}$ are orthogonal subsets of an inverse semigroup
so too is $\{a_{1}b_{1}, \ldots a_{i}b_{j}, \ldots, a_{m}b_{n}\}$.
It follows that $S''$ is closed under multiplication.
If $\{a_{1},\ldots, a_{m}\}$ is an orthogonal subset of an inverse semigroup so too is $\{a_{1}^{-1},\ldots, a_{m}^{-1}\}$.
Thus if $a_{1} + \ldots + a_{m} \in S''$ then  $a_{1}^{-1} + \ldots + a_{m}^{-1} \in S''$.
Observe that 
$$(a_{1} + \ldots + a_{m})(a_{1}^{-1} + \ldots + a_{m}^{-1}) = a_{1}a_{1}^{-1} + \ldots + a_{m}a_{m}^{-1}$$
and
$$(a_{1}^{-1} + \ldots + a_{m}^{-1})(a_{1} + \ldots + a_{m}) = a_{1}^{-1}a_{1} + \ldots + a_{m}^{-1}a_{m}$$
and so
$(a_{1} + \ldots + a_{m})(a_{1}^{-1} + \ldots + a_{m}^{-1})(a_{1} + \ldots + a_{m})
=
a_{1} + \ldots + a_{m}$.
We have therefore shown that $S''$ is a regular semigroup.
To show that $S''$ is inverse, it is enough to prove that the idempotents in $S''$
are precisely the elements of the form $e_{1} + \ldots + e_{m}$ where $e_{1}, \ldots, e_{m}$
are idempotents in $S'$ and form an orthogonal subset.
Suppose that $\sum_{i=1}^{m} a_{i}$ is an idempotent in $S''$ where $\{a_{1}, \ldots, a_{m}\}$ is an orthogonal subset of $S'$.
Then
$$\left( \sum_{i=1}^{m} a_{i} \right)^{2} = \sum_{i=1}^{m} a_{i}.$$
Multiply both sides of this equation on the left by $a_{1}a_{1}^{-1}$.
Then
$$a_{1} = a_{1}^{2} + a_{1}a_{2} + \ldots + a_{1}a_{m}.$$
Now multiply both sides of this equation on the right by $a_{1}^{-1}a_{1}$.
It follows that $a_{1} = a_{1}^{2}$.
By symmetry, it follows that each of the elements $a_{1}, \ldots, a_{m}$ is an idempotent.
Thus $a_{1} + \ldots + a_{m}$ is an idempotent.
We have therefore proved that $S''$ is an inverse monoid with zero.

%%%%%%%%%%%%%%%%%%%%%%%%%%%%%%%%%%%%%%%%%%%%%%%%%%%%%%%%%%%%%%%%%%%%%%%%%%%%%%%%%%%%%%%%%%%%%%%%%%%%%%%%%%%%%%
By Lemma~\ref{lem:dingbat}, to prove that $S''$ is a Boolean inverse monoid
it is enough to prove that $\mathsf{E}(S'')$ is closed under the operation $\mathbf{e} \mapsto 1 - \mathbf{e}$.
Referring back to the proof of Lemma~\ref{lem:dingbat}, we see that if $e,f \in \mathsf{E}(S)$ then
$e \circ f = e(1-f) + f$ which is an element of $\mathsf{E}(S'')$.
Let $e_{1}, \ldots, e_{m} \in \mathsf{E}(S)$.
Define $[e_{1}, \ldots, e_{m}] = (\ldots ((e_{1} \circ e_{2}) \circ e_{3}) \ldots \circ e_{m})$;
in other words, associate to the left.
Then  $[e_{1}, \ldots, e_{m}] \in \mathsf{E}(S'')$.
We now have the following two identities.
Let $e_{1}, \ldots, e_{m} \in \mathsf{E}(S)$.
Then
$$1 - \left(e \prod_{i=1}^{m}(1 - e_{i})\right) = (1 - e) +e[e_{1}, \ldots, e_{m}].$$
This can be proved from the following identity
$$\prod_{i=1}^{m}(1 - e_{i}) = 1 - [e_{1}, \ldots, e_{m}],$$
which can be proved by induction.
Let $\mathbf{e}_{1}, \ldots, \mathbf{e}_{m}$ be an orthogonal set of elements in $\mathsf{E}(S')$.
Then
$$1 - \left( \sum_{i=1}^{m} \mathbf{e}_{i} \right) = \prod_{i=1}^{m} (1 - \mathbf{e}_{i}),$$
where the proof is straightforward.
\end{proof}

If $S \subseteq R$ define $\mathsf{B}_{R}(S) = S''$ above which we call the {\em Booleanization of $S$ in $R$}.

Wehrung \cite[Theorem 6-1.7]{Wehrung2} gives an alternative construction in the case of rings with involution
which can be traced back to the work of Renault \cite{Renault, Renault2} on Cartan subalgebras of $C^{\ast}$-algebras as well as Kumjian \cite{K2}.
This is more adapted to representations of inverse semigroups by partial isometries in $C^{\ast}$-algebras which we shall return
to at the end of this section.

Let $R$ be a unital ring and let $S$ be an inverse monoid with zero.
A {\em representation} of $S$ in $R$ is a homomorphism of monoids $\theta \colon S \rightarrow R$
which maps zero to zero.
The image of $\theta$ is an inverse semigroup and so there is a Boolean inverse monoid $\mathsf{B}_{R}(\mbox{\rm im}(\theta)) \subseteq R$.
Thus, by restricting the codomain of $\theta$ and by a mild abuse of notation,
$\theta \colon S \rightarrow \mathsf{B}_{R}(\mbox{\rm im}(\theta))$.
By Theorem~\ref{them:bc}, there is a morphism $\theta^{\ast} \colon \mathsf{B}(S) \rightarrow \mathsf{B}_{R}(\mbox{\rm im}(\theta))$.
Thus, by extending the codomain of $\theta^{\ast}$ and by a mild abuse of notation, $\theta^{\ast} \colon \mathsf{B}(S) \rightarrow R$.
However, $\theta^{\ast}$ has the additional property that if $a,b \in  \mathsf{B}(S)$ are orthogonal then
$\theta^{\ast}(a \vee b) = \theta^{\ast}(a) + \theta^{\ast}(b)$.
Let $S$ be a Boolean inverse monoid and $R$ a ring.
Then a representation $\phi \colon S \rightarrow R$ is called an {\em additive representation}
if $a \perp b$ in $S$ implies that $\phi (a \vee b) = \phi (a) + \phi (b)$.
We have therefore proved the following.

\begin{proposition}\label{prop:swell} Let $\theta \colon S \rightarrow R$ be a
representation of the inverse monoid $S$ in the unital ring $R$.
Then there is a unique additive representation $\theta^{\ast} \colon \mathsf{B}(S) \rightarrow R$
such that $\theta^{\ast} \beta = \theta$.
\end{proposition}

The case where the ring is actually a $C^{\ast}$-algebra is of particular interest.
There are then some minor modifications to the definitions.
Let $S$ be an inverse semigroup and $C$ a $C^{\ast}$-algebra.
A representation $\theta \colon S \rightarrow C$ is a {\em $\ast$-representation} if $\theta (s^{-1}) = \theta (a)^{\ast}$.
This implies that $S$ is being represented by partial isometries in the $C^{\ast}$-algebra.
The following is almost immediate from the above calculations with obvious ammendments.

\begin{proposition}\label{lem:sue} Let $\theta \colon S \rightarrow C$ be a $\ast$-representation of the inverse semigroup $S$ in the $C^{\ast}$-algebra $C$.
Then there is a unique additive $\ast$-representation $\Theta \colon \mathsf{B}(S) \rightarrow C$ such that $\theta = \Theta \beta$.
\end{proposition}

Let $S$ be an inverse semigroup with zero.
We can construct the {\em contracted semigroup algebra} $\mathbb{C}_{0}S$.
The following is a re-interpretation of what Paterson proves.
It is proved by combining \cite[Proposition 4.4.3]{Paterson} with Theorem~\ref{them:paterson}.

\begin{theorem}\label{them:russia} Let $S$ be an inverse semigroup with zero.
Then $\mathsf{B}(S) \cong \mathsf{B}_{\mathbb{C}_{0}S}(S)$.
\end{theorem}

%POLY
%%%%%%%%%%%%%%%%%%%%%%%%%%%%%%%%%%%%%%%%%%%%%%%%%%%%%%%%%%%%%%%%%%%%%%%%%%%%%%%%%%%%%%%%%%%%%%%%%%%%%%%%%%%%%%
\subsection{The Booleanization of the polycyclic monoids $P_{n}$}

In this section, we shall carry out an explicit computation of the Booleanization of an important family of inverse semigroups.
Our computations will rely on the perspective provided by Section~2.4.
Our inverse semigroups will actually be monoids, but Remarks~\ref{rem:monoid} and \ref{rem:more-monoid} tell us that
this will be taken care of in our construction.\footnote{A key part of our computation is the description of the Stone space of the set of all reverse definite languages over a fixed alphabet.
This space is actually described in \cite{Pin2017} though for completely different reasons from ours.}

Let $A = \{a_{1}, \ldots, a_{n}\}$, where $n \geq 2$ and finite, and 
denote by $A^{\ast}$ the free monoid generated by $A$ with concatenation as its multiplication.
Any subset of $A^{\ast}$ is called a {\em language (over $A$)}.
If $x,y \in A^{\ast}$ we write $x \leq_{p} y$ if and only if $x = yz$ for some string $z$.
We say that $y$ is a {\em prefix} of $x$.
We call $\leq_{p}$ the {\em prefix ordering}.
A {\em prefix code} is a set of finite strings which are pairwise $\leq_{p}$-incomparable.
A maximal prefix code is a prefix code that cannot be a proper subset of another prefix code.
A subset $R \subseteq A^{\ast}$ is called a {\em right ideal} if $RA^{\ast} \subseteq R$.
For each right ideal $R$, there is a unique prefix code $P$ such that $R = PA^{\ast}$ \cite[Lemma A.1]{Birget}.
We say that $R$ is {\em finitely generated} if $P$ is finite.
The intersection of two finitely generated right ideals of $A^{\ast}$ is also finitely generated \cite[Lemma~3.3]{Birget}.
A right ideal $R$ is said to be {\em essential} if $R$ has a non-empty intersection with every right ideal.
The essential right ideals are precisely those right ideals $PA^{\ast}$ where $P$ is a maximal prefix code \cite[Lemma A.1]{Birget}. 
The following result is important.

\begin{lemma}\label{lem:cornwall} Let $R = PA^{\ast}$ be a finitely generated right ideal of $A^{\ast}$.
Then $R$ is essential if and only if $A^{\ast} \setminus R$ is finite.
\end{lemma}
\begin{proof} Suppose first that $R$ is essential.
Then $P$ is a finite maximal prefix code.
By \cite[Lemma~A1(4)]{Birget}, it is immediate that $A^{\ast} \setminus PA^{\ast}$ is a finite set.
Conversely, suppose that $A^{\ast} \setminus R$ is finite.
Let $x$ be any finite string.
The set $xA^{\ast}$ is infinite so cannot be disjoint from the set $R$ since the complement of $R$ is finite.
It follows that $xA^{\ast} \cap R \neq \varnothing$.
\end{proof}

Let $R_{1}$ and $R_{2}$ be right ideals of $A^{\ast}$.
A {\em morphism} $\theta \colon R_{1} \rightarrow R_{2}$ is a function such that $\theta (xu) = \theta (x)u$ for all $x \in R_{1}$ and $u \in A^{\ast}$.
Let $\theta \colon P_{1}A^{\ast} \rightarrow P_{2}A^{\ast}$ be an ismorphism where $P_{1}$ and $P_{2}$ are prefix codes.
Then the restriction of $\theta$ to $P_{1}$ induces a bijection $\mathsf{T}(\theta) \colon P_{1} \rightarrow P_{2}$ called the {\em table} of $\theta$.
More generally, any bijection from $P_{1}$ to $P_{2}$ is called a {\em table}.
There is a bijection between isomorphisms $P_{1}A^{\ast} \rightarrow P_{2}A^{\ast}$ and tables $P_{1} \rightarrow P_{2}$.

It is now easy to show that the subset $R(A^{\ast})$ of $I(A^{\ast})$ consisting of all isomorphisms 
between the finitely generated right ideals of $A^{\ast}$ is a distributive inverse monoid.

The {\em polycyclic monoid} $P_{n}$, where $n \geq 2$,  is the monoid with zero given by the following monoid presentation
$$P_{n} = \langle a_{1}, \ldots, a_{n}, a_{1}^{-1}, \ldots, a_{n}^{-1} \colon a_{i}^{-1}a_{i} = 1 \text{ and } a_{i}^{-1}a_{j} = 0 \text{ if } i \neq j \rangle.$$
We refer the reader to \cite{Law2007} for all the details and \cite[Section~9.3]{Law1} (though beware that the notation is slightly different in the latter).
The goal of this section is to compute $\mathsf{B}(P_{n})$.
The first step is to compute the distributive completion of $P_{n}$.
This was, in fact accomplished in \cite{Law2007}, 
where we proved that $\mathsf{D}(P_{n})$ is isomorphic to the inverse monoid of right ideal isomorphisms between the finitely generated right ideals of $A^{\ast}$.
In fact, we proved a slightly different theorem there so we shall first explain why in the case of the polycyclic inverse monoids it accomplishes what we claim.
There, we in fact constructed {\em orthogonal completions}.
However, polycyclic inverse monoids have a special property that ensures orthogonal completions and distributive completions are the same thing;
an inverse semigroup $S$ is said to be {\em ramified}\footnote{We have borrowed this terminology from \cite{Heindorf}.} if for $a,b,c \in S$  $a \leq b,c$ implies that $b \leq c$ or $c \leq b$.

\begin{lemma}\label{lem:ramified} 
The polycyclic inverse monoids are ramified.
\end{lemma}
\begin{proof} Suppose that $yx^{-1} \leq vu^{-1}, zw^{-1}$.
Then $(y,x) = (v,u)p = (z,w)q$ for some $p,q \in A^{\ast}$.
Then $y = vp = zq$ and $x = up = wq$.
The strings $v$ and $z$ are prefix comparable.
Without loss of generality suppose that $v = zr$ for some string $r$.
Then $q = rp$ and $u = wr$.
Thus $(v,u) = (z,w)r$ and so $v^{-1}u \leq zw^{-1}$.
\end{proof}

The significance of being ramified is explained by the following lemma.

\begin{lemma}\label{lem:orthogonal} Let $S$ be a ramified inverse semigroup.
Let $A = \{a_{1}, \ldots, a_{n}\}^{\downarrow}$ be any finitely generated compatible order ideal.
Then $A = \{b_{1}, \ldots, b_{m} \}^{\downarrow}$ where $\{b_{1}, \ldots, b_{m} \}$ is an orthogonal set.
\end{lemma}
\begin{proof} Consider the element $a_{1}$ and any $a_{i}$ where $2 \leq i \leq n$.
Since $a_{1} \sim a_{i}$ their meet exists.
Suppose that $a_{1} \wedge a_{i} = 0$.
Then by Lemma~\ref{lem:macron}, we deduce that $a_{1}$ and $a_{i}$ are orthogonal.
Suppose that $a_{1} \wedge a_{i} \neq 0$.
Then since $S$ is ramified either $a_{1} \leq a_{i}$ or $a_{i} \leq a_{1}$.
Without loss of generality, suppose the former.
Then $a_{1}$ may be discarded.
This process can be repeated and we obtain in this way a subset of $\{a_{1}, \ldots, a_{n}\}$
which is orthogonal and still generates $A$.
\end{proof}

It now follows, as claimed, that \cite{Law2007} establishes $\mathsf{D}(P_{n})$ as precisely the set of all isomorphisms between finitely generated right ideals of $A^{\ast}$.
The idempotents of $\mathsf{D}(P_{n})$ form a distributive lattice isomorphic to the distributive lattice of finitely generated right ideals of $A^{\ast}$ under subset inclusion.
The intersection of two finitely generated right ideals is a finitely generated right ideal
and the union of two finitely generated right ideals is a finitely generated right ideal.
It follows that $\mathsf{D}(P_{n})$ is a distributive subalgebra of $I(A^{\ast})$.
Thus by Section~2.4, to compute the Booleanization of $\mathsf{D}(P_{n})$ it will be enough to compute the Boolean hull of $\mathsf{D}(P_{n})$ in $I(A^{\ast})$.
To do this, it is convenient to use terminology and notation from language theory \cite{Pin}.

I will use regular expressions to describe languages so $+$ means $\cup$ and singleton sets are denoted by their elements.
A language $L$ over $A$ is said to be {\em definite}\footnote{Strictly speaking, this should be {\em reverse definite}.} if $L = X + YA^{\ast}$ where both $X$ and $Y$ are finite languages.
It is well-known from the theory of regular languages \cite{Pin} that the set of definite languages in $A^{\ast}$ forms a Boolean algebra with respect to set intersection, union and complementation.

\begin{lemma} 
The set of definite languages is generated as a Boolean algebra by the finitely generated right ideals of $A^{\ast}$.
\end{lemma}
\begin{proof} Denote by $\mathscr{B}$ the Boolean subalgebra of $\mathsf{P}(A^{\ast})$, the power set of $A^{\ast}$, generated by the finitely generated right ideals of $A^{\ast}$.
Observe that $\{x\} = xA^{\ast} \setminus xAA^{\ast}$.
Thus $\mathscr{B}$ contains all finite languages and so all unions of finite languages and finitely generated right ideals.
Thus $\mathscr{B}$ contains all definite languages.
But the set of definite languages is a Boolean algebra.
\end{proof}

It follows that the set of definite languages is the Booleanization of the distributive lattice of finitely generated right ideals.
We shall construct a Boolean inverse submonoid of $I(A^{\ast})$ whose Boolean algebra of idempotents is isomorphic to the set of definite languages over $A$.
Before we do that, it is useful to make some simple observations about definite languages.

An element $x \in L$ of a definite language is said to be {\em unbounded} if $xA^{\ast} \subseteq L$ otherwise it is said to be {\em bounded}.
Every definite language $L$ can be written as a disjoint union $L = X_{1} + X_{2}$ where $X_{1}$ are the bounded elements of $L$ and $X_{2}$ are the unbounded elements.
The set $X_{1}$ is finite and the set $X_{2}$ is a finitely generated right ideal.
The set $X_{2}$ has a minimum generating set which is a prefix code \cite{Birget}. 
We say that a definite language $L$ is in {\em normal form} if it is written $L = X + YA^{\ast}$ where $YA^{\ast}$ are all the unbounded elements, $Y$ is a prefix code, and $X$ are all the bounded elements.

\begin{example}{\em The following is adapted from \cite{Br}.
Let $L = 0 + 201 + 212 + (00+ 20+ 01 + 02)(0+1+2)^{\ast}$.
This is a definite language.
We now convert it into normal form.
We show first that $0$ is unbounded.
Observe that $0(0+1+2)^{\ast} = 0 + 00 + 01 + 02 + (00 + 01 + 02)(0 + 1 + 2)^{\ast}$.
It follows that
$L = 201 + 212 + (0 +00 + 20+ 01 + 02)(0+1+2)^{\ast}$.
But $202$ is unbounded because $20(0+1+2)^{\ast} \subseteq L$.
It follows that
$L = 212 + (0 +00 + 20+ 01 + 02 + 201)(0+1+2)^{\ast}$.
Now we observe that $(0 +00 + 20+ 01 + 02 + 201)(0+1+2)^{\ast} = (0 + 20)(0+1+2)^{\ast}$.
Thus $L = 212 + (0 + 20)(0+1+2)^{\ast}$, which is in normal form.}
\end{example}

The proof of the following is now routine.

\begin{lemma} 
Two definite languages are equal if and only if their normal forms are the same.
\end{lemma}

Let $L_{1}$ and $L_{2}$ be definite languages.
A bijection $\alpha \colon L_{1} \rightarrow L_{2}$ is said to be {\em permissible} if it satisfies the following two conditions:
\begin{enumerate}

\item $\alpha$ maps bounded elements to bounded elements and unbounded elements to unbounded elements.

\item If $x \in L_{1}$ is an unbounded element and $y \in A^{\ast}$ is arbitrary then $\alpha (xy) = \alpha (x)y$.

\end{enumerate}

For convenience, we list the notation we shall be using:
\begin{itemize}

\item $I(A^{\ast})$ is the symmetric inverse monoids of all partial bijections on the set $A^{\ast}$.

\item $I_{f}(A^{\ast})$ is the inverse semigroup of all partial bijections between the finite subsets of $A^{\ast}$.

\item $R(A^{\ast})$ is the inverse semigroup of all isomorphisms between finitely generated right ideals of $A^{\ast}$.

\item $CT(A^{\ast})$ is the set of all permissible maps between definite languages.
Clearly, the idempotent elements here are the identity functions on the definite languages.

\end{itemize}

It follows that each permissible map is a disjoint union of an element of $I_{f}(A^{\ast})$ and an element of $R(A^{\ast})$.

Let $S$ be a distributive inverse semigroup.
An {\em additive ideal} $I$ of $S$ is a semigroup ideal which is also closed under binary compatible joins.
The proof of the following is immediate.

\begin{lemma}\label{lem:stuff} 
$I_{f}(A^{\ast})$ is an additive ideal of $I(A^{\ast})$.
\end{lemma}

The proof of the following is also straightforward.

\begin{lemma}\label{lem:duke} Let $U$ be a distributive inverse semigroup.
Let $S$ and $T$ be distributive inverse subsemigroups where both are closed under binary compatible joins
and where $S$ is an (additive) ideal.
Put $V = \{s \vee t \colon s \in S, t \in T, s \perp t \}$.
Then $V$ is a distributive inverse subsemigroup of $U$.
\end{lemma}

\begin{proposition}\label{prop:lord} 
$CT(A^{\ast})$ is a Boolean inverse $\wedge$-monoid.
\end{proposition}
\begin{proof} By Lemma~\ref{lem:stuff}, $CT(A^{\ast})$ is a distributive inverse monoid 
with a Boolean algebra of idempotents and so is a Boolean inverse monoid.
It remains to show that it has all binary meets.
This is equivalent \cite{Leech} to proving the following.
Let $\alpha \colon L \rightarrow M$ be a permissible map between definite languages.
Define $\mbox{Fix}(\alpha) = \{x \in L \colon \alpha (x) = x\}$.
Then $F$ is a definite language.
If $F = \varnothing$ then we are done,
so in what follows we can assume that $F \neq \varnothing$.
Let $L = L_{1} + L_{2}A^{\ast}$ be the normal form of $L$.
We prove that 
$$\mbox{Fix}(\alpha) = \mbox{Fix}(\alpha | L_{1}) + \mbox{Fix}(\alpha | L_{2})A^{\ast}.$$
It is clear that the right hand side is contained in the left hand side.
Observe that if $x$ is an unbounded element of $L$ and is fixed by $\alpha$ then $\alpha$ also fixes all elements of $xA^{\ast}$.
We prove that the left hand side is contained in the right hand side.
Let $x \in \mbox{Fix}(\alpha)$ be unbounded.
Then $x \in L_{2}A^{\ast}$.
We can therefore write $x = py$ where $p \in L_{2}$.
We have that $\alpha (x) = \alpha (py) = \alpha (p)y$.
But by assumption $\alpha (x) = x$.
Thus $\alpha (p) = p$.
It follows that $p \in \mbox{Fix}(\alpha | L_{2})$ and so $x \in \mbox{Fix}(\alpha | L_{2})A^{\ast}$.
If $x$ is bounded then it is immediate that $x \in \mbox{Fix}(\alpha | L_{1})$. 
\end{proof}

%%%%%%%%%%%%%%%%%%%%%%%%%%%%%%%%%%%%%%%%%%%%%%%%%%%%%%%%%%%%%%%%%%%%%%%%%%%%%%%%%%%%%%%%%%%%%%%%%%%%%%%%%%%%%%%%%%%%%%%%%%%%%%%%%
We now come to our main theorem.

\begin{theorem}\label{them:TC} 
The Boolean inverse monoid $CT(A^{\ast})$ is the Booleanization of the polycyclic inverse monoid $P_{n}$.
\end{theorem}
\begin{proof} This is almost immediate by the results of Section~2.4.
It devolves to checking that the Boolean hull of $P_{n}$ in $I(A^{\ast})$ is in fact $CT(A^{\ast})$. 
Clearly, the idempotents of $\mathsf{B}_{I(A^{\ast})}(P_{n})$ are the same as the idempotents of $CT(A^{\ast})$
and $\mathsf{B}_{I(A^{\ast})}(P_{n}) \subseteq CT(A^{\ast})$ since $\mathsf{D}(P_{n}) \subseteq CT(A^{\ast})$. 
An element of $CT(A^{\ast})$ is an orthogonal join of an isomorphism between two finitely generated right ideals of $A^{\ast}$ and a bijection between two finite subsets of $A^{\ast}$.
This latter map is itself an orthogonal join of the maps that take one element sets to one element sets.
Let $u,v$ be two strings.
Once we have shown that the partial bijection $u \mapsto v$ belongs to $\mathsf{B}_{I(A^{\ast})}(P_{n})$, our proof will be complete.
Define $f \colon uA^{\ast} \rightarrow vA^{\ast}$ given by $f(ux) = vx$.
Define $g \colon uAA^{\ast} \rightarrow vAA^{\ast}$ given by $g(uax) = vax$ where $a \in A$.
It is clear that $g \leq f$ in the natural partial order.
Observe that $f \setminus g$ is precisely the map $u \mapsto v$. 
\end{proof}

%%%%%%%%%%%%%%%%%%%%%%%%%%%%%%%%%%%%%%%%%%%%%%%%%%%%%%%%%%%%%%%%%%%%%%%%%%%%%%%%%%%%%%%%%%%%%%%%%%%%%%%%%%%%%%%%%%%%%%%%%%%%%%%%%
We call the Boolean inverse monoid $CT(A^{\ast})$ the {\em Cuntz-Toeplitz monoid (of degree $n$)} \cite{JSW}.
The rationale for this terminology will now be explained.
We denote by $C_{n}$ the {\em Cuntz inverse monoid} \cite{Law2007b} and use the description given in \cite[Section~5.2]{LS}.
Denote by $A^{\omega}$ the set of all right-infinite strings over the alphabet $A$.
The monoid $C_{n}$ consists of all bijections $f \colon XA^{\omega} \rightarrow YA^{\omega}$, where $X$ and $Y$ are prefix codes,
for which there exists an associated bijection $f_{1} \colon X \rightarrow Y$ such that $f(xw) = f_{1}(x)w$ where $x \in X$ and $w \in A^{\omega}$.

The behaviour of $\wedge$-morphisms between Boolean inverse $\wedge$-semigroups  is analogous to that of the behaviour of homomorphisms between rings.
Let $\theta \colon S \rightarrow T$ be a $\wedge$-morphism between two Boolean inverse $\wedge$-semigroups.
Define the {\em kernel} of $\theta$, denoted by $\mbox{ker}(\theta)$, to be the set of all $s \in S$ such that $\theta (s) = 0$. 
It is easy to check that $\mbox{ker}(\theta)$ is an additive ideal of $S$.

\begin{lemma}\label{lem:kernel} 
Let $\theta$ and $\phi$ be two surjective $\wedge$-morphisms between the Boolean inverse $\wedge$-semigroups  $S$ and $T$.
Then $\theta = \phi$ if and only if  $\mbox{\rm ker}(\theta) = \mbox{\rm ker}(\phi)$.
\end{lemma}
\begin{proof} Suppose that $\theta (a) = \theta (b)$.
Then $\theta (a \setminus (a \wedge b)) = 0 =  \theta (b \setminus (a \wedge b))$.
By assumption $\phi (a \setminus (a \wedge b)) = 0 =  \phi (b \setminus (a \wedge b))$.
Thus 
$$\phi (a) = \phi ((a \setminus (a \wedge b) \vee (a \wedge b)) = \phi (a \wedge b).$$
By symmetry, we get that $\phi (a) = \phi (b)$ and symmetry again delivers the result.
\end{proof}

In the light of the above lemma, we may extend the usual notation from ring theory.
Let $S$ be a Boolean inverse $\wedge$-semigroup and let $I$ be an additive ideal of $S$.
Denote by $S/I$ the Boolean inverse $\wedge$-semigroup $S/\varepsilon_{I}$ where
$\varepsilon_{I}$ is the congruence defined by $(a,b) \in \varepsilon_{I}$ if and only if 
$a\setminus (a \wedge b), b \setminus (a \wedge b) \in I$.

\begin{proposition}\label{prop:CT} 
$CT(A^{\ast})/I_{f}(A^{\ast}) \cong C_{n}$.
\end{proposition}
\begin{proof}  Denote by $\equiv$ the congruence relation induced on $CT(A^{\ast})$ by the additive ideal $I_{f}(A^{\ast})$
We prove, first, the following.
Suppose that $R$ and $R'$ are two finitely generated right ideals of $A^{\ast}$
and that $1_{R} \cong 1_{R'}$.
Then $R$ is essential if and only if $R'$ is essential.
Suppose that $R$ is essential.
Then $A^{\ast} \setminus R$ is a finite set by Lemma~\ref{lem:cornwall}.
It follows that $1_{A^{\ast}} \cong 1_{R}$.
Thus $1_{A^{\ast}} \cong 1_{R'}$.
From the definition, $A^{\ast} \setminus R'$ is a finite set and so by Lemma~\ref{lem:cornwall},
$R'$ is also essential.
It is immediate from the above that  $1_{A^{\ast}} \cong 1_{R}$ precisely when $R$ is essential.
This is already enough to tell us that $CT(A^{\ast})/I_{f}(A^{\ast})$ is a homomorphic image of $C_{n}$ by \cite{Law2007b}.
But $C_{n}$ is congruence-free.
Thus as long as the quotient is not trivial it will be isomorphic to $C_{n}$.
But this is clear.

More concretely, we may also prove the result as follows.
Let $f \colon (X_{1} + Y_{1}A^{\ast}) \rightarrow (X_{2} + Y_{2}A^{\ast})$ be a permissible map where $Y_{1}$ and $Y_{2}$ are prefix codes.
Then $f$ induces a bijection $f_{1} \colon Y_{1} \rightarrow Y_{2}$.
Define $\Theta (f) \colon Y_{1}A^{\omega} \rightarrow Y_{2}A^{\omega}$ by $\Theta (f)(yw) = f_{1}(y)w$.
It is clear that $\Theta$ is a surjective $\wedge$-morphism (and a monoid homomorphism)
and that the kernel of $\Theta$ is $I_{f}(A^{\ast})$.
\end{proof}

By Proposition~\ref{prop:CT}, it follows that $CT(A^{\ast})$ is the Boolean inverse monoid  analogue of the {\em Cuntz-Toeplitz algebra}.
See \cite{JSW}, for example.

%BIB
%%%%%%%%%%%%%%%%%%%%%%%%%%%%%%%%%%%%%%%%%%%%%%%%%%%%%%%%%%%%%%%%%%%%%%%%%%%%%%%%%%%%%%%%%%

\end{document}